\documentclass[10pt]{article}
\usepackage[pagebackref,colorlinks=true,pdfpagemode=none,urlcolor=blue,linkcolor=blue,citecolor=blue]{hyperref}
\usepackage{amsthm, amssymb, bm}
\usepackage{soul, color}
\usepackage[]{amsmath}
\usepackage[]{amsfonts}
\usepackage[]{tikz-cd}
\usepackage[]{graphicx}
\usepackage{todonotes}
\newtheorem{Theorem}{Theorem}[section]
\newtheorem{Lemma}[Theorem]{Lemma}
\newtheorem{Corollary}[Theorem]{Corollary}
\newtheorem{Proposition}[Theorem]{Proposition}
\newtheorem{Conjecture}[Theorem]{Conjecture}
\theoremstyle{definition}
\newtheorem{Definition}[Theorem]{Definition}
\theoremstyle{remark}
\newtheorem{Remark}[Theorem]{Remark}

\usepackage{changes}
\definechangesauthor[name=Francois, color=red]{F}

\def\t{\tilde{\tau}}
\def\Sm{\mathbb{S}}

\def\Rm{ {\mathbb{R}} }
\def\Cm{ {\mathbb{C}} }
\def\Nm{ {\mathbb{N}} }
\def\Zm{ {\mathbb{Z}} }
\def\A{\mathcal{A}}
\def\D{\mathcal{D}}
\def\H{\mathcal{H}}
\def\V{\mathcal{V}}

\def\two{\mathrm{II}}
\newcommand{\calE}{\mathcal E}

\newcommand{\bt}{ {\boldsymbol{t}}}
\newcommand{\be}{ {\boldsymbol{e}}}

\newcommand{\adj}{ I_0^\sharp }

\newcommand{\wtB}{ \widetilde{B} }
\newcommand{\wtH}{ \widetilde{H} }
\newcommand{\hpi}{ \widehat{\pi} }

\newcommand{\cout}[1]{ }

\newcommand{\del}{\partial}
\newcommand{\phg}{\mathrm{phg}}

\newcommand{\F}[1]{{#1}}
\newcommand{\resp}[1]{{#1}}

\newcommand{\foot}{ {\pi_{\mathrm f}} }
\newcommand{\hF}{ \widehat{\foot} }

\newcommand {\dM}{\mathrm{d}M}

\title{Double b-fibrations and desingularization of the X-ray transform on manifolds with strictly convex boundary}
\author{Rafe Mazzeo\thanks{Department of Mathematics, Stanford University, Stanford CA 94305; email:rmazzeo@stanford.edu} \and Fran\c{c}ois Monard\thanks{Department of Mathematics, University of California, Santa Cruz CA 95064; email:fmonard@ucsc.edu}}
\date{}

\begin{document}
\maketitle

\begin{abstract}
We study the mapping properties of the X-ray transform and its adjoint on spaces of conormal functions on Riemannian 
manifolds with strictly convex boundary.  After desingularizing the double fibration, and expressing the X-ray transform and its adjoint using 
b-fibrations operations, \resp{we employ tools related to Melrose's Pushforward Theorem to describe the mapping properties of these operators 
on various classes of polyhomogeneous functions, with special focus to computing how leading order coefficients are transformed.  
The appendix explains that a na\"ive use of the Pushforward Theorem leads to a suboptimal result with non-sharp index sets. 
Our improved results are obtained by closely inspecting} Mellin functions which arise in the process, showing that certain coefficients vanish. 
This recovers some sharp results known by other methods. A number of consequences for the mapping properties of the X-ray transform and its 
normal operator(s) follow.
\end{abstract}

\tableofcontents

\section{Introduction}\label{sec:intro}
In this article we study the mapping properties of the geodesic X-ray transform and its adjoint on manifolds with convex boundary.

To set the stage and introduce the main objects of study, let $(M^d,g)$ be a Riemannian manifold with strictly convex boundary (in the sense of 
having positive definite second fundamental form). Denote by $SM\stackrel{\pi}{\longrightarrow} M$ its unit tangent bundle. The boundary of 
$SM$ splits into outward (-) and inward (+) boundaries:
\begin{align}
    \partial SM = \partial_+ SM \cup \partial_- SM, \qquad \partial_\pm SM := \{ (x,v)\in SM,\ x\in \partial M,\ \pm g_x(v,\nu_x)\ge 0 \}, 
    \label{eq:dSM}
\end{align}
with $\nu_x$ the unit {\it inward} normal at $x\in \partial M$.  Note that $\partial_+ SM$ is naturally diffeomorphic to
$\partial_- SM$ under the reflection involution $(x,v)\mapsto (x,-v)$, and furthermore that $\partial_+ SM \cap \partial_-SM$ is naturally identified
with $S\partial M$; this is called the glancing set.  We also denote by $\varphi_t \colon SM\to SM$ the geodesic flow. This is
defined for $(x,v)\in SM$ and $t\in [-\tau(x,-v),\tau(x,v)]$, where $\tau(x,v) \in [0,\infty]$ is the smallest time for which 
$\varphi_t(x,v) \in \partial SM$. Throughout, we will assume that $(M,g)$ is {\em non-trapping} in the sense that 
$\sup_{SM} \tau <\infty$. The footpoint map 
\begin{align}
    \foot\colon SM\ni (x,v) \mapsto \varphi_{-\tau(x,-v)} (x,v) \in \partial_+ SM,
    \label{eq:F}
\end{align}
assigns to any point and direction the initial (or 'final in backward time') point on the boundary and inward-pointing direction determined by the geodesic through $(x,v)$; 
this is continuous on $SM$ and smooth on $(SM)^{o}$ (Here and below the superscript $o$ denotes the interior of the set in question.). We then define the geodesic X-ray transform $I_0$ and the backprojection operator $\adj$: 
\begin{align}
    I_0 f (x,v) &= \int_0^{\tau(x,v)} f(\pi (\varphi_t(x,v)))\ dt, \qquad (x,v) \in \partial_+ SM, \label{eq:forward} \\
    \adj g (x) &= \int_{S_x} g(\foot (x,v)) dS_x(v), \qquad x \in M. \label{eq:adjoint}
\end{align} 
The operator $I_0$ is continuous from $L^2(M, \dM)$ to $L^2(\partial_+ SM, \mu\, d\Sigma^{2d-2})$, where 
\[
\mu(x,v) = g_x(v,\nu_x),
\]
\resp{$\dM$ is the Riemannian density on $M$, } and $d\Sigma^{2d-2}$ is the hypersurface measure inherited by $\partial_+ SM$ from the Sasaki volume form $d\Sigma^{2d-1}$ on $SM$. Given these choices of measure, $\adj $ is the Hilbert space adjoint of $I_0$. 

Recent works \cite{Monard2017,Monard2019,Monard2020a} justify the necessity of addressing mapping properties of $\adj  I_0$ in a way
that does not require use of an extension of $M$. For manifolds with strictly geodesically convex boundaries, the results in the literature are sparse. These offer a variety of unrelated insights: 
\begin{itemize}
\item Pestov and Uhlmann construct in \cite{Pestov2005} a space $C_{\alpha}^\infty(\partial_+ SM)$ (see Eq.\,\eqref{eq:Cpm} below) 
of functions whose extensions as first integrals of the geodesic flow are smooth on $SM$. This space satisfies
\begin{align}
	    \adj (C_\alpha^\infty(\partial_+ SM)) \subset C^\infty(M).
	    \label{eq:PU}
\end{align}
They go on to characterize $C_\alpha^\infty (\del_+ SM)$ in terms of the scattering relation using fold theory. We also refer to
\cite{Paternain2021} for a more recent take on this issue.
\item In \cite{Monard2017}, the authors show that $\adj  I_0$ extends to a larger manifold engulfing $M$; assuming it 
satisfies a $-\tfrac12$-transmission condition at $\partial M$, they conclude that 
	\begin{align}
	    \adj  I_0 \colon \left\{
		\begin{array}{lll}
		    d^{-1/2} C^\infty(M) & \stackrel{\cong}{\longrightarrow} & C^\infty(M), \\
		    H^{(-1/2), s}(M) & \stackrel{\cong}{\longrightarrow}  & H^{s+1}(M) \qquad (s>-1),
		\end{array}
		\right.
		\label{eq:MNPresult}
	    \end{align}
are isomorphisms. Here $H^{(-1/2), \bullet}(M)$ is the scale of H\"ormander transmission spaces, and $H^{\bullet}(M)$ the classical Sobolev 
scale. Even though \eqref{eq:MNPresult} implies a set of estimates, these mapping properties do not fully predict the mapping
properties of iterates of the operator $\adj  I_0$, which must be understood when using Landweber's iteration for purposes of inversion. 
\item In \cite{Monard2019a} it is shown that in the class of simple geodesic disks with constant curvature, the operator $\adj  \, \mu^{-1} I_0$ is 
an isomorphism of $C^\infty(M)$. This is done by constructing a Sobolev scale $\wtH^k(M)$, defined using a degenerate elliptic operator, 
which have intersection $C^\infty(M)$ and which satisfy 
	    \begin{align}
		\adj  \, \mu^{-1} I_0 \colon \wtH^k(M) \stackrel{\cong}{\longrightarrow} \wtH^{k+1}(M).
		\label{eq:Mresult}
	    \end{align}
In other words, by introducing the singular weight $\frac{1}{\mu}$, one can identify spaces which are stable under iteration of the modified 
normal operator $\adj  \, \mu^{-1} I_0$. 
\item In \cite{Katsevich2001}, the mapping properties for the Euclidean dual Radon transform (i.e., integration over hyperplanes) are given.
These allow one to understand how index sets on $\Rm\times \Sm^{n-1}$ translate into index sets on $\Rm^n$. In particular, this explains 
much about the Euclidean X-ray transform on the plane. 
\end{itemize}

The present work presents a first step towards a unified understanding of these various results about mapping properties of $I_0$, $\adj $ 
and the normal operators built out of them, at the level of Fr\'echet spaces.

As is well known in integral geometry, it is very helpful to recast the problem in terms of double fibrations, cf.\ Helgason's work 
on symmetric spaces \cite{Helgason2010}, Guillemin's realization of integral-geometric operators as FIO's \cite{Guillemin1979}, and 
the formulation of the Bolker condition \cite{Monard2013b,Holman2015}, which leads to a more precise understanding of the impact 
of interior conjugate points on the X-ray transform. A famous and basic example of double fibrations in the homogeneous space setting 
is that of Euclidean motions on $\Rm^2$ via centralizers of points and lines:
\begin{align*}
\Rm^2 \longleftarrow E(2) \longrightarrow \Rm\times\Sm^1.
\end{align*}
Using this, the two-dimensional X-ray transform can be viewed as the composition of pullback by the map onto $\Rm^2$ followed by
pushforward by the map onto $\Rm \times \Sm^1$. For the geodesic X-ray transform on non-homogeneous Riemannian manifolds, 
one can construct a {\em local} double fibration out of the point-geodesic relation \cite{Monard2013b,Gelfand1969}, or a double fibration 
of $SM$ by (base-) points and geodesics. However, if $M$ has a strictly convex boundary, such fibrations cannot be extended globally
due to the presence of 'short' geodesics. The difficulty arises because as $(x,v)\in SM$ approaches $S(\partial M)$, the maximal geodesic 
passing through $(x,v)$ collapses to a point. Because of this, the extensions to $SM$ of certain smooth functions on $\del_+ SM$
as first integrals of geodesic flow are neither smooth nor even conormally regular.  In this article, we show that this difficulty
can be `resolved' (in two senses of the word): there is a desingularization by means of a map 
\begin{align}
    \Upsilon \colon \partial_+ SM\times [0,1]\ni ( (x,v) ,u) \mapsto \varphi_{u\tau(x,v)}(x,v) \in SM,
    \label{eq:Upsilon}
\end{align}
recently introduced in \cite{Monard2020a}, which gives rise to a double $b$-fibration in the sense of geometric microlocal 
analysis \cite{Melrose1992,Grieser2001} (see Lemma \ref{lem:bfibrations} below), represented by the dashed lines in 
Figure \ref{fig:doublefib}. 
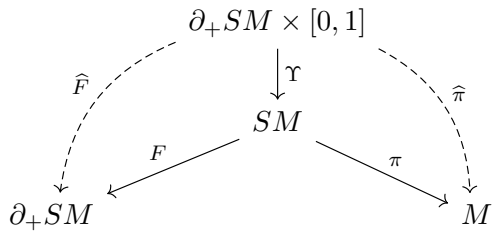
\begin{figure}[htpb]
    \centering
    \begin{tikzcd}[]
    & \partial_+ SM \times [0,1]  \arrow{d}{\Upsilon} \arrow[dashed, bend left]{ddr}{\widehat{\pi}} \arrow[dashed, bend right]{ddl}[swap]{\widehat{\foot}} & \\  &  SM \arrow{dl}[swap]{\foot}\arrow{dr}{\pi} & \\
\partial_+ SM && M
\end{tikzcd}
    \caption{Desingularization of the double fibration into a double b-fibration via the map $\Upsilon$}
    \label{fig:doublefib}
\end{figure}

\resp{Moreover, the operators $I_0$ and $\adj$ can be written as weighted compositions of pushforward and pullback by the maps $\hpi:= \pi\circ\Upsilon$ and $\hF:= \foot\circ\Upsilon$. In this setup, one can then in principle apply Melrose's pullback and pushforward Theorems to understand the mapping properties of $I_0$ and $\adj$ on smooth or polyhomogeneous functions. However, this process (carried out in Appendix \ref{sec:PFT} for comparison to our main results) produces polyhomogeneous functions which apparently have index sets, i.e., the set of exponents in their expansion, significantly larger than desired. This is less sharp than the result obtained by Pestov and Uhlmann \eqref{eq:PU}. 

Our main results, Theorems \ref{thm:Calphamapping} and \ref{thm:I0starIDX} below, provide mapping properties for $I_0$ and $\adj$ which recover the result by Pestov and Uhlmann, while extending the functional setting to spaces with polyhomogenous behavior at the boundaries of $M$ and $\partial_+ SM$. Notably, much of the work is in the study of $\adj$, which displays a variety of interesting phenomena (e.g., index set creation or cancellation, see Theorem \ref{thm:I0starIDX}) depending on the index set of the input function. To obtain such refinements, we first introduce} new singular coordinates on $\partial_+ SM$ (or more accurately, we blow up certain $p$-submanifolds in this space) which have the effect of resolving the scattering relation. We also need to examine the mechanism behind the pushforward theorem in greater detail, i.e., study the Mellin functionals, whose poles and residues give the index set and associated leading coefficients.  We can then see that certain poles do not in fact occur. 
This cancellation is ultimately due to fact that the Beta function,
\begin{align}
    B(z,z) = \int_0^1 (u(1-u))^{z-1}\, du, 
    \label{eq:Beta}
\end{align}
(and appropriate generalizations of it) has \resp{zeros} at the negative half integers $-\Nm_0-1/2$. Another important part of this is the introduction of a parity condition on the Taylor series (resulting in $C_\alpha^\infty$-index sets below, see Sec. \ref{sec:prelim}). Such parity conditions have been used in other applications of geometric microlocal analysis, in particular in local index theory and spectral geometry \cite{Melrose1993,Mazzeo2014}, asymptotics of Poincar\'e-Einstein metrics \cite{Fefferman1985} and renormalized volume \cite{Albin2009,Bahuaud2019}, analytic continuation and resolvent estimates for the Laplacian on asymptotically hyperbolic spaces \cite{Vasy2013,Vasy2017}, and inverse problems related to the above contexts \cite{MarxKuo2024,Eptaminitakis2021}.

In the next section, we formulate the main results before discussing some consequences, including refined boundary determination of polyhomogeneous integrands, and natural candidates for Fr\'echet spaces where it should be possible to formulate stable mapping properties of normal operators (i.e., weighted compositions of $I_0$ and $\adj$). We give an outline of the remainder of the article at the end of the next section.

\section{Statement of the main results} \label{sec:main}

\F{After first defining in Section \ref{sec:prelim} appropriate regularity classes and notions of index set (unrestricted on $M$, and with parity constraints on $\partial_+ SM$), Section \ref{ssec:main} gives the main results, Theorems \ref{thm:Calphamapping} and \ref{thm:I0starIDX},} describing how polyhomogeneous functions are transformed under the operators $I_0$ and $\adj $; in other words, we determine how these operators act on index sets. \F{In Section \ref{sec:consequences}, we formulate and prove some consequences of the main results. We also propose Conjecture \ref{conj} on the mapping properties of normal operators in singularly weighted situations.}

\subsection{Preliminaries: polyhomogeneous functions on manifolds with boundary} \label{sec:prelim}

Let $N^n$ be a manifold with boundary, and choose adapted coordinates $(x,y)$ near $\partial N$, where $x \geq 0$ is a boundary defining function (bdf) for $\del N$ and $y = (y_1, \dots, y_{n-1})$ restricting to a coordinate system on the boundary.  Although $C^\infty(N)$ is the most obvious class of `nice' functions on $N$, it is not the most natural in many analytic and geometric problems.  There are two standard characterizations of $C^\infty$:  smooth functions enjoy {\it stable regularity}, i.e., maintain regularity upon repeated differentiation, or that they admit Taylor series expansions.  Each of these can be generalized in slightly different ways, the first by restricting the types of differentiations which are allowed, and the second by allowing slightly more general asymptotic expansions as $x \to 0$.  These lead to the class of conormal functions and to the notion of polyhomogeneity, respectively.

The key to finding a suitable generalization of stable regularity is to weaken the set of allowable differentiations. Accordingly, we define $\mathcal V_b(N)$ to be the space of smooth vector fields on $N$ which are unconstrained in the interior but which lie tangent to the boundary. With $(x, y)$ a local coordinate system as above, we have
\begin{equation}
\mathcal V_b(M) = C^\infty-\mathrm{span}\, \{ x\del_x, \del_{y_1}, \ldots, \del_{y_{n-1}} \}.
\label{Vb}
\end{equation}

Now define the space of conormal functions of order $s \in \mathbb R$ on $N$ by 
\begin{align}
    \begin{split}
	\A^s (N) := \{u\in C^\infty(N^{o}) &: V_1 \ldots V_k u \in x^s L^\infty(N)  \\
	& \text{ for all } V_j \in \mathcal V_b(N)\ \mbox{and for all}\ k \in \mathbb N_0\}. 	
    \end{split}
    \label{eq:conormal}
\end{align}
(Here and below the superscript $o$ denotes the interior of the set in question.) Note that $\A^s(N)$ contains functions of the form $x^{s'}|\log x|^a$ for any $s'$ with $\mathrm{Re}\, s' > s$, $a \in \mathbb R$ (or $\mathrm{Re}\, s' \geq s$ if $a \leq 0$).   More generally, $\A^s(N)$ contains functions which are discrete or continuous superpositions of these basic ones. As a rule of thumb, conormal functions are completely smooth in tangential directions and have special regularity properties normal to the boundary. The intersection $\cap_{s\in \Rm} \A^s (N)$ is just the space $\dot{C}^\infty (N)$ of smooth functions vanishing to infinite order at $\partial N$. Here and below, we denote $\A^\infty(N) := \cup_{s\in \Rm} \A^s(N)$.

We next define the particularly nice families of polyhomogeneous functions, which are subspaces of the space of conormal functions.
Polyhomogeneous function are conormal, but in addition have classical asymptotic expansions involving 
powers $x^z (\log x)^k$ ($z \in \mathbb C$ and $k \in \mathbb N_0$) as $x \to 0$.  A particular space of polyhomogeneous
functions is determined by its index set $E$, which is the set of exponent pairs $(z,k)$ which
can appear in the expansions of its elements.  By definition, an index set $E$ is a subset of $\Cm\times \Nm_0$ with the following properties: 

\resp{
\begin{enumerate}
	\item[(a)] For every $s\in \Rm$, the set $E_{\le s} := \{(z,k)\in E : \text{Re} (z)\le s\}$ is finite. 
	\item[(b)] If $(z,k)\in E$, then $(z,\ell)\in E$ for any $\ell = 0,1,\dots,k-1$.
\end{enumerate}
If $E$ also satisfies
\begin{enumerate}
	\item[(c)] If $(z,k)\in E$, then $(z+1,k)\in E$,
\end{enumerate}
then it is called a $C^\infty$ index set.

Given any index set $E$, define 
\begin{equation}
    \A_{\mathrm{phg}}^E\resp{(N)} = \left\{u \in \A^\infty(N): u \sim \sum_{(z,k) \in E}  u_{z,k}(y) x^z (\log x)^k,\ \ u_{z,k} \in C^\infty(\del N) \right\}.
	\label{phg} 
\end{equation}
This asymptotic expansion is `classical' in the sense that the difference between $u$ and any finite partial sum of this series decays as $x \searrow 0$ 
like the next term in the series, and with corresponding decay properties for any derivative of $u$ and the corresponding term-by-term derivative of the series.

Condition (a) implies the existence of a real number $\inf E := \inf \{\text{Re} (z),\ (z,k)\in E\}$ such that $\A_{\mathrm{phg}}^E\resp{(N)} \subset \A^s(N)$ 
for all $s<\inf E$. If $k=0$ for all $(z,k)\in E$ with $\text{Re}(z) = \inf E$, then $\A_{\mathrm{phg}}^E\resp{(N)} \subset \A^{\inf E}(N)$.  
If condition (c) is satisfied, then the moniker $C^\infty$ index set ensures that $\A_{\mathrm{phg}}^E(N)$ is a $C^\infty(N)$ module and is 
invariant under any smooth change of variables of $N$. 

As examples, $\A_{\mathrm{phg}}^{\Nm_0\times \{0\}}(N) = C^\infty(N)$ and $\A_{\mathrm{phg}}^{\Nm_0\times \{0,1\}}(N) = C^\infty(N)+\log x \, C^\infty(N)$.

\paragraph{Even structures: $C_\alpha^\infty$-index sets on $\partial_+ SM$.} When the manifold $N$ has additional symmetries, the spaces $\A_{\phg}^E$ 
may not encode these symmetries, and refinements of the objects above are needed. In our setting, the manifold of interest is $\partial_+ SM$, 
as defined in \eqref{eq:dSM}; it has boundary
\begin{align}
	\partial_0 SM = \partial_+ SM\cap \partial_- SM = \{(x,v)\in SM:\ x\in \partial M,\ g_x(v,\nu_x) = 0\}.
	\label{eq:d0SM}
\end{align}
When $(M,g)$ is non-trapping with strictly convex boundary, then $\partial SM$ carries a natural involution called the scattering relation 
$\alpha\colon \partial SM\to \partial SM$. This is defined by
\begin{align}
	\alpha(x,v) = \varphi_{\pm \tau(x,\pm v)} (x,v), \quad (x,v) \in \partial_\pm SM;
	\label{eq:scatrel}
\end{align}
it leaves each $(x,v)\in \partial_0 SM$ fixed.   Using this involution, one may define the important subspaces of $C^\infty(\partial_+ SM)$
\begin{align}
C_{\alpha,\pm}^\infty (\partial_+ SM) := \{u\in C^\infty(\partial_+ SM),\quad A_\pm u \in C^\infty(\partial SM)\},
\label{eq:Cpm}
\end{align}
where 
\begin{align*}
A_\pm u(x,v) := \left\{
    \begin{array}{ll}
	u(x,v), & (x,v)\in \partial_+ SM \\
	\pm u(\alpha(x,v)), & (x,v)\in \partial_- SM
    \end{array}
    \right.
\end{align*}
for $u\in C^\infty(\partial_+ SM)$. 
(Note that $C_{\alpha,+}^\infty$ coincides with the space $C_\alpha^\infty(\partial_+ SM)$ described in the introduction.) 
Functions in $C_{\alpha,+/-}^\infty (\partial_+ SM)$ are smooth in the interior of $\partial_+ SM$ with Taylor expansions at $\partial_0 SM$ which 
involve only even/odd terms. This makes sense only if these expansions are written with respect to a restricted class of `even' boundary defining functions,
described below. To state our main results, we introduce the key definitions and relegate details to Section \ref{sec:indexSets}. 

On $\partial_+ SM$, a bdf $\bt$ for $\partial_0 SM$ is called an $\alpha$-boundary defining function ($\alpha$-bdf) if moreover $\bt\in C_{\alpha,-}^\infty(\partial_+ SM)$. A good example of an $\alpha$-bdf is the exit time function $\tau\colon \partial_+ SM\to \Rm$. A $C_\alpha^\infty$-index set on $\partial_+ SM$ is an index set as defined via conditions (a)-(b) above, and where condition (c) is replaced by 
\begin{enumerate}
	\item[(c')] If $(z,k)\in E$, then $(z+2,k)\in E$.
\end{enumerate}
Given a $C_\alpha^\infty$-index set $E$, we then say that $u\in \A_{\phg,\alpha}^E(\partial_+ SM)$ if $u\in C^\infty( (\partial_+ SM)^o)$ and if for some $\alpha$-bdf $\bt$, we have the asymptotic expansion off of $\partial_0 SM$
\begin{align}
	u\sim \sum_{(z,k)\in E} \bt^z (\log \bt)^k u_{z,k},\quad u_{z,k}\in C^\infty(\partial_0 SM).
	\label{eq:Aphgalpha}
\end{align}
In Section \ref{sec:indexSets}, we establish that such a definition is invariant under change of $\alpha$-bdf and stable under multiplication by elements of $C_{\alpha,+}^\infty(\partial_+ SM)$. As notable examples, $C_{\alpha,+}^\infty(\partial_+ SM) = \A_{\phg,\alpha}^{2\Nm_0\times \{0\}} (\partial_+ SM)$ and $C_{\alpha,-}^\infty(\partial_+ SM) = \A_{\phg,\alpha}^{(2\Nm_0+1)\times \{0\}} (\partial_+ SM)$.

}

\subsection{Main results} \label{ssec:main}

\paragraph{Weighted X-ray transforms.} 

Let $(M,g)$ be convex and non-trapping, let $\phi\in C^\infty(SM)$ be a smooth weight function, and for $f\in C_{\resp{c}}^\infty((SM)^{o})$, define
\begin{align*}
I^\phi f (x,v) = \int_0^{\tau(x,v)} \phi(\varphi_t(x,v)) f(\varphi_t (x,v))\ dt, \qquad (x,v) \in \partial_+ SM. 
\end{align*}
Then $I^\phi f\in C^\infty( (\partial_+ SM)^{o})$ and our first result, \resp{whose proof is given Sec. \ref{sec:proofmainXray},} shows that if $f$ is \resp{polyhomogeneous}, then $I^\phi f$ is also polyhomogeneous with index set determined by that of $f$.  

\begin{Theorem} \label{thm:Calphamapping}
Let $(M,g)$ be convex, non-trapping, and let $\phi\in C^\infty(SM)$. Suppose $f \in \A^E_{\phg}(SM)$ for some \resp{$C^\infty$-}index set $E$ with 
$s = \inf(E)>-1$. Then $I^\phi f \in  \A^F_{\phg\F{,\alpha}}(\del_+ SM)$, where $F$ is a $C^\infty_\alpha$-index set contained in 
(and possibly equal to) $\{ (2z+1,\ell), \ (z,\ell)\in E \}$.

More precisely, fix $\gamma\in \Cm$ \resp{with $\text{Re}(\gamma)>-1$}, $k\in \Nm_0$, $h\in C^\infty(\partial SM)$ and $\chi\in C_c^\infty([0,\infty))$ a cutoff function equal to $1$ near $0$.
Using coordinates $(\rho,\xi)\in [0,\varepsilon)\times \partial SM$ on $SM$ near $\partial SM$ and $(\tau,\omega)$ on $\partial_+ SM$ near 
$\partial_0 SM$ with $\omega = (y,w)\in \partial_0 SM$, as defined in Lemma \ref{lem:scatcoords}, and denoting by $\tau$ the geodesic length, we have 
\begin{align*}
	I^\phi [\F{\chi(\rho)}\rho^\gamma \log^k \rho\ h(\xi)] (\tau,\omega) \sim  \tau^{2\gamma+1} \sum_{p\ge 0 \atop 0 \le \ell \le k} \tau^{2p} (\log\tau)^\ell 
a_{2\gamma+1+p,\ell}(\omega).
\end{align*}
If $\rho(x) = d_g(x,\partial M)$ near $\partial M$, the coefficient in front of the most singular term equals 
    \begin{align}
	    a_{2\gamma+1,k}(\omega) = 2^k h(\omega) \phi(\omega) \two(\resp{\omega})^\gamma\ B(\gamma+1,\gamma+1), \qquad \resp{\omega\in \partial_0 SM,}
	\label{eq:I0BottomTerm}
    \end{align}
    where $\two(\omega) := \two_y(\resp{w},\resp{w})$ is the second fundamental form, and $B$ is the Beta function \eqref{eq:Beta}.
\end{Theorem}
Letting $\phi \equiv 1$ and supposing that $h(\xi)$ is independent of the tangent variable $v$ (i.e., replace $\xi\in \partial SM$ by $y\in \partial M$)
this result includes the operator $I_0$.  This result generalizes \cite[Proposition 6.13]{Monard2020a}. 

\paragraph{Backprojection operator.} Now turn to the backprojection operator $\adj $.  As an elementary input function, consider 
$a(\omega) \resp{\chi(\tau) \tau^\gamma \log^k \tau}$, where $\omega\in\partial_0 SM$, 
$a(\omega)$
is smooth, $\chi \in C_c^\infty([0,\infty))$ is a cutoff function equal to $1$ near $0$, and $\tau$ is the geodesic length. \resp{Here and below, for $y\in \partial M$, $dA(y)$ denotes the Riemannian hypersurface density on $\partial M$ and $dS_y$ denotes the Euclidean measure on the unit sphere $S_y (\partial M)$. Our second main result, whose proof is given in Section \ref{sec:proofbackprojection}, is the following.}

\begin{Theorem}\label{thm:I0starIDX}
	Fix $\gamma\in \Cm$, $k\in \Nm_0$ and $a\in C^\infty(\partial_0 SM)$. Then $\adj  (a(\omega) \resp{\chi(\tau) \tau^\gamma \log^k \tau} )$ is polyhomogeneous,
    with index set contained in (and possibly equal to) the union \F{$\Nm_0\times \{0\} \cup E_{\gamma,k}$, where 
    \begin{align}
	E_{\gamma,k} := \left\{
	    \begin{array}{ll}
		(\frac{\gamma+1}{2} + \Nm_0) \times \{0, \dots, k-1\}, & \gamma\in 2\Nm_0, \\
		(\frac{\gamma+1}{2} + \Nm_0) \times \{0, \dots, k\} \cup ( \Nm_0 \cap (\frac{\gamma+1}{2} + \Nm_0) ) \times \{k+1\}, & \gamma\in 2\F{\Zm}+1, \\
		(\frac{\gamma+1}{2} + \Nm_0) \times \{0, \dots, k\}, & \text{otherwise},
	    \end{array}
	    \right.
	    \label{eq:Egamk}
	\end{align}
	where in the first case, $E_{\gamma,k} = \emptyset$ if $k=0$.}

	In coordinates $(\rho,y)\in [0,\varepsilon)\times \partial M$ on $M$ near $\partial M$ 
	    with $\rho(x) = d_g(x,\partial M)$, we conclude that
	    \begin{align*}
		\adj  (a(\omega) \resp{\chi(\tau) \tau^\gamma \log^k \tau} ) \sim \sum_{(z,\ell)\in \Nm_0 \cup E_{\gamma,k}} a_{z,\ell}(y) \rho^z \log^\ell \rho,
	    \end{align*}
	    where the lowest order coefficient in the part of the expansion corresponding to $E_{\gamma,k}$ is given by
	    \begin{align*}
		a_{z_0,\ell_0}(y) = c_{z_0,\ell_0} \int_{S_y (\partial M)} a(y,w) \two_y(w)^{\frac{-\gamma+1}{2}}\ dS_y(w), \qquad \omega = (y,w)\in \partial_0 SM. 		
	    \end{align*}
	    Here 
	    \begin{equation}
		\begin{cases}
		    \gamma=2m \Rightarrow (z_0,\ell_0) = (m+1/2,k-1), \  &c_{z_0,\ell_0} = \frac{2\pi k}{2^{k+3}} \frac{4^{2m+2}}{2m+2} \binom{2m+2}{m+1}^{-1},\ m\in\Nm_0,  \\[0.5ex]
		    \gamma=2m\F{-1} \Rightarrow (z_0,\ell_0) = (\F{m}, k+1),  & c_{z_0,\ell_0} = -\binom{\F{2m}}{\F{m}}/((k+1) 2^{k+2}), \ \  \F{m\in\Nm_0}, \\[0.5ex]
		    \gamma\notin \Nm_0\F{-1} \Rightarrow  (z_0,\ell_0) = (\frac{\gamma+1}{2},k), & c_{z_0,\ell_0} = {2^{-(k+3)}} 
		    B\left(-\frac{\gamma+1}{2}, -\frac{\gamma+1}{2} \right).  
		\end{cases}
		\label{eq:Agk}
	    \end{equation}
\end{Theorem}

The other coefficients of the expansion with exponents in $E_{\gamma,k}$ are locally determined, and in principle computible, while the coefficients in the part of the expansion with coefficients in $\Nm_0\times\{0\}$ are non-local.

\resp{
    
    \begin{Remark}
	To put the first half of Theorem \ref{thm:I0starIDX} in invariant terms, fix $E$ a $C_{\alpha}^\infty$-index set. By invariance of $\A_{\phg,\alpha}^E (\partial_+ SM)$ with respect to an $\alpha$-bdf, any element of $h\in \A_{\phg,\alpha}^E (\partial_+ SM)$ may be written as an asymptotic sum in the $\alpha$-bdf function $\tau$, and applying Theorem \ref{thm:I0starIDX} to each term separately yields
	\begin{align}
	    I_0^\sharp \left( \A_{\phg,\alpha}^E (\partial_+ SM) \right) \subset \A_{\phg}^F (M), \qquad \text{where} \quad F:= \Nm_0 \times \{0\} \cup \bigcup_{(\gamma,k)\in E} E_{\gamma,k}, 
	    \label{eq:mappingI0sharp}
	\end{align}
	with $E_{\gamma,k}$ defined in \eqref{eq:Egamk}. One easily verifies that $F$ is a $C^\infty$-index set. 
    \end{Remark}    
}

\F{
\begin{Remark}
	As the scheme of proof of Theorem \ref{thm:I0starIDX} is based on a correspondence between index sets of a phg function and the set of poles of its associated 'Mellin functional' (see \ref{sec:Mellin}), the first case of Eq. \eqref{eq:Egamk} is an instance where poles cancel at the Mellin level. This case recovers the result \eqref{eq:PU}, which a standard use of the Push-Forward Theorem could not recover (see the discussion at the end of Appendix \ref{sec:PFT}).  

    The second case in Eq. \eqref{eq:Egamk} is an instance where poles of higher order (or, equivalently, higher powers of $\log \rho$ in expansion terms) become created. This only affects the most singular term in the expansion if $\gamma\in 2\Nm_0-1$, hence the discrepancy between the cases of \eqref{eq:Egamk} and those of \eqref{eq:Agk}.
\end{Remark}
}

\subsection{Consequences of main results} \label{sec:consequences}

\subsubsection{Boundary determination}

\begin{Theorem}[Boundary determination for functions on $M$]\label{thm:bdrydet}
Let $\phi\in C^\infty(SM)$ be such that for every $y\in \partial M$, there exists some $w\in S_y(\partial M)$ with $\phi(y,w)\ne 0$. 
Suppose $f\in \A^E_{\phg}(M)$ for some index set $E$ with $\inf (E)>-1$. 

(i) If $I^\phi f \in \A^{1+2s'} (\partial_+ SM)$ for some $s'>\inf E$, then $f\in \A^{E_{\geq s'}}_{\phg}(M)$. 

    (ii) If $I^\phi f \in \dot{C}^\infty (\partial_+ SM)$, then $f\in \dot{C}^\infty(M)$.
\end{Theorem}

\begin{proof}[Proof of Theorem \ref{thm:bdrydet}]
	For (i), suppose that $I^\phi f \in \A^{1+2s'} (\partial_+ SM)$ with $s' > \inf(E)$, and assume by contradiction that the most singular term in the expansion of $f$ is of the form $h(\xi) \rho^\gamma (\log \rho)^k$ with $Re(\gamma) < s'$. By assumption, the corresponding coefficient $a_{1+2\gamma, k}(\omega)$, as given by \eqref{eq:I0BottomTerm}, must vanish for all $\omega = (y,w)\in \partial_0 SM$. Note that in the case where $f$ is a function on $M$, $h$ only depends on $y$ and not on $w$. Fixing an arbitrary $y\in \partial M$, we evaluate the product \eqref{eq:I0BottomTerm} at an $w\in S_y (\partial M)$ such that $\phi(y,w) \ne 0$. There, we have $B(\gamma+1,\gamma+1) \neq 0$ since $Re(\gamma)>-1$, and $\two_y(w,w) > 0$ by convexity, and hence we must have $h(y) = 0$. 

	Notice that this result is local: if $I^\phi f \in \A^{1+2s'}$ only in a neighborhood of some point $y_0$, then the coefficients of the terms in the expansion of $f$ with $Re(z) < s'$ must vanish in that neighborhood.

	The proof of (ii) follows from (i) and the fact that $\dot{C}^\infty(M) = \cap_{s\in \Rm} \A^s(M)$. \resp{It is likely that this follows from other techniques, but we include this here for completeness.}
\end{proof}

The same proof (without the assumption that the coefficient $h$ in the proof above only depends on the basepoint) yields the \resp{following result.}

\begin{Theorem}[Boundary determination for functions on $SM$]\label{thm:bdrydetSM}
Let $\phi\in C^\infty(SM)$ and let $N_\phi \subset \partial_0 SM$ be the set of points of $\partial_0 SM$ where $\phi$ is nonzero. 
Suppose $f\in \A^E_{\phg}(SM)$ for some $C^\infty$ index set $E$ with $\inf(E)>-1$. 

(i) If $I^\phi f \in \A^{1+2s'} (\partial_+ SM)$ for some $s'>\inf E$, then for all \resp{$(\gamma,k)\in E$ such that $Re(\gamma)<s'$}, $f_{\gamma,k}$ vanishes on $N_\phi$. 

    (ii) If $I^\phi f \in \dot{C}^\infty (\partial_+ SM)$, then for all $(\gamma,k)\in E$, $f_{\gamma,k}$ vanishes on $N_\phi$.
\end{Theorem}

\subsubsection{Mapping properties}

\paragraph{The operator $\adj  I_0$.}
The following result explains why iterative use of $\adj  I_0$ is inherently unstable in terms of fine regularity
properties.

\begin{Corollary} \resp{Let $(M,g)$ be a convex, non-trapping Riemannian manifold of dimension $\ge 2$, and let $\rho$ be a boundary defining function.} For any $k\in \Nm_0$, 
    \begin{align*}
	(\adj  I_0)^k (C^\infty(M)) \subset \sum_{\ell=0}^k (\resp{\rho}\log \resp{\rho})^\ell C^\infty(M),
    \end{align*}   
    and there are functions $f \in C^\infty(M)$ for which $(\adj  I_0)^k f$ include $(\resp{\rho}\log \resp{\rho})^k$ in their expansion.
\end{Corollary}

\begin{proof}
    We prove the following statement by induction on $k\ge 0$. Defining
    \begin{align}
	E_k := \{ (z,p), z\in \Nm_0,\ 0\le p\le \min (z,k) \},
	\label{eq:Ek}
    \end{align}
    if $f$ has index set $E_k$, then $\adj  I_0 f$ has index set $E_{k+1}$. This follows immediately from Theorem \ref{thm:Calphamapping} along with the second case (pole creation) of Theorem \ref{thm:I0starIDX}. The second assertion follows from the explicit formul\ae\ for leading coefficients. 
\end{proof}

Recalling the additional mapping property $\adj  I_0 (\resp{\rho}^{-1/2} C^\infty (M))\subset C^\infty(M)$, we deduce the sequence of spaces, each continuously mapped into the next by $\adj  I_0$: 
\begin{align}
	\resp{\rho}^{-1/2} C^\infty(M) \to C^\infty(M) \to {\cal A}_{\resp{\phg}}^{E_1} (M) \to \dots \to {\cal A}_{\resp{\phg}}^{E_k} (M) \to \dots,
\label{eq:seqspaces}
\end{align}
where $E_k$ is defined in \eqref{eq:Ek}.  It is natural to wonder if there is a space which extends the sequence to the left, i.e., which maps to $\resp{\rho}^{-1/2} C^\infty(M)$. If it were to exist, such a space would {\it not} contain only polyhomogeneous functions since 
\begin{itemize}
    \item Using \eqref{eq:combinatorics} below, we see that $I_0 \rho^z = \tau^{2z+1} \int_0^1 (u(1-u))^z F^z\ du$, which makes sense only when $Re(z)>-1$. 
    \item With this integrability condition in mind, $\adj  I_0 \rho^z$ has index set contained in $(\Nm_0 \times \{0\})\, \overline{\cup} \, ((z+\Nm)\times \{0\})$, and the minimum index has real part $0$. In particular, $\adj  I_0$ can never generate terms which blow-up near $\partial M$. 
\end{itemize}

\paragraph{The operator $\adj  \, \tau^{-1} I_0$.}
We now discuss another example which shows that the various cases in Theorem \ref{thm:I0starIDX} can interact `cyclically' (in contrast to \eqref{eq:seqspaces}). 

\begin{Corollary}
Let $(M,g)$ be a convex, non-trapping Riemannian manifold. Then 
\begin{align*}
C^\infty(M) + \log\rho\ C^\infty(M) &\overset{\adj  \, \tau^{-1} I_0} \longrightarrow  \ldots\\
&C^\infty(M) + \rho^{1/2} C^\infty (M)  \overset{\adj  \, \tau^{-1} I_0}\longrightarrow C^\infty(M) + \log\rho\ C^\infty(M).
\end{align*}    
\end{Corollary}
The first map annihilates the log terms through pole cancellation (first case of Theorem \ref{thm:I0starIDX}), while the second map re-creates them (second case of Theorem \ref{thm:I0starIDX}).

\paragraph{Singularly weighted normal operators.}

There is already a precedent for finding (singularly) weighted versions for the Euclidean Radon transform where the Singular Value Decomposition 
is sometimes computable, see e.g.\, \cite{Louis1984}. The two-dimensional examples there, applied to the X-ray transform, motivate a more 
general class of singularly weighted functional settings. 

\begin{Proposition}\label{prop:weightedI0}
Let $(M,g)$ be a convex, non-trapping Riemannian manifold with boundary and bdf $\rho$, let $\bt$ an $\alpha$-bdf on $\partial_+ SM$ and fix $\gamma\in\Rm$. Then 
for $\delta<2\gamma+1$, the operator $I_0 \rho^\gamma$, which sends $f$ to $I_0 (\rho^\gamma f)$, yields a bounded map
\begin{align}
I_0 \rho^{\gamma} \colon L^2(M, \rho^{\delta}\dM) \to L^2(\partial_+ SM, \bt^{2\delta-4\gamma-1}\,  \mu \, d\Sigma^{2d-2}). 
\label{eq:weightedI0}
\end{align}
\end{Proposition}

\F{The proof of Proposition \ref{prop:weightedI0} is given Section \ref{sec:proofmainXray}.} A special case of Proposition \ref{prop:weightedI0} is when $\resp{\delta} = \gamma > -1$. In this case,
\begin{align*}
    I_0 \rho^{\gamma} \colon L^2(M, \rho^{\gamma}\dM) \to L^2(\partial_+ SM, \bt^{-2\gamma-1} \, \mu \, d\Sigma^{2d-2}),
\end{align*}
is bounded, \resp{and its adjoint relative to these Hilbert spaces is $\adj  \bt^{-2\gamma-1}$, as the following computation shows, valid for all $f\in C^\infty(M)$, $h \in C^\infty(\partial_+ SM)$:
\begin{align*}
	\int_{\partial_+ SM} (I_0 \rho^\gamma f) h \bt^{-2\gamma-1} \mu\ d\Sigma^{2d-2} = \int_{M} \rho^\gamma f \adj (h \bt^{-2\gamma-1}) \dM.
\end{align*}
Above we have used that $I_0^\sharp$ is the $L^2(M,\dM)\to L^2(\partial_+ SM, \mu d\Sigma^{2d-2})$ adjoint of $I_0$ (see e.g. \cite[Lemma 4.1.4]{Paternain2021}).} In particular, the operators 
\begin{align*}
    \adj  \bt^{-2\gamma-1} I_0 \rho^\gamma \colon  L^2(M, \rho^{\gamma}\dM) \to  L^2(M, \rho^{\gamma}\dM), \qquad \gamma > -1
\end{align*}
are all self-adjoint. Moreover, these operators are {\em isomorphisms} of $C^\infty(M)$ in the following cases: 
\begin{itemize}
    \item $\gamma = 0$, $(M,g)$ a simple geodesic disk of constant curvature and $\bt = \mu$ (in circularly symmetric cases, $\mu$ is an $\alpha$-bdf), cf.\ \cite{Monard2019a}.
    \item $\gamma = -1/2$ and $(M,g)$ is a simple surface \resp{(as formulated in \eqref{eq:MNPresult}), cf. \cite{Monard2019}}. 
    \item \F{$\gamma>-1$ and $(M,g)$ is a simple geodesic disk in a constant curvature model, cf. \cite{Mishra2022}.}
\end{itemize}

From Theorem \ref{thm:Calphamapping}, we deduce the following 
\begin{Lemma}
    For any $\gamma>-1$, any bdf $\rho$ and $\alpha$-bdf $\bt$, the operator $\adj  \bt^{-2\gamma-1} I_0 \rho^\gamma$ maps $C^\infty(M)$ into itself.
\end{Lemma}
\begin{proof} If $f\in C^\infty(M)$, then for $\gamma>-1$, $\rho^\gamma f$ has index set $(\gamma + \Nm_0)\times \{0\}$ and by 
Theorem \ref{thm:Calphamapping}, $I_0 (\rho^\gamma f) \sim \bt^{2\gamma+1} \sum_{p\ge 0} \bt^{2p} a_p(\omega)$ for some coefficients
$a_p\in C^\infty(\partial_0 SM)$. Thus $\bt^{-2\gamma-1} I_0 (\rho^\gamma f)\in C_{\alpha,+}^\infty(\partial_+ SM)$. Applying $\adj$ 
and using \eqref{eq:PU}, we obtain the result.
\end{proof}

It is natural to ask whether this is an isomorphism for all $\gamma>-1$. 
\begin{Conjecture} \label{conj}
	Let $(M,g)$ be a simple Riemannian manifold with bdf $\rho$\F{, and fix $\gamma>-1$. Then,} there exists an $\alpha$-bdf $\bt$ for $\partial_+ SM$ such that the operator $\adj  \bt^{-2\gamma-1} I_0 \rho^\gamma$ is an isomorphism of $C^\infty(M)$. 
\end{Conjecture}

It is easily seen that if the $C^\infty(M)$-isomorphism property holds, it holds independently of the choice of bdf $\rho$. However, its dependence on the choice of $\alpha$-bdf $\bt$ (except for $\gamma=-1/2$, where no $\alpha$-bdf appears) remains unclear. 

Unless $\gamma= -1/2$, the operator here is not extendible across $\del M$, and thus we cannot invoke the transmission conditions of Boutet de Monvel. 

\resp{

\paragraph{Structure of the article.} The remainder of the article is organized as follows. 

In Section \ref{sec:doubleb}, we first carry out the desingularization of the double fibration into a double b-fibration (see Lemma \ref{lem:bfibrations}) and show in Lemma \ref{lem:Irewrite} that the operators $I_0$ and $I_0^\sharp$ can be written as weighted compositions of pushforwards and pullbacks by these b-fibrations. As a necessary ingredient to our main theorems, this justifies that $I_0$ and $I_0^\sharp$ map conormal spaces to conormal spaces. 

In Section \ref{sec:glancingCoords}, we construct special coordinates on $\partial_+ SM$ with good parity properties with respect the scattering relation (an important tool in the proof of Theorem \ref{thm:I0starIDX}), before covering $C_\alpha^\infty$-index sets and polyhomogenous spaces on $\partial_+ SM$ in detail in Section \ref{sec:indexSets}.

Section \ref{sec:proofmain} then covers the proofs of the main results: the proofs of Theorem \ref{thm:Calphamapping} and Proposition \ref{prop:weightedI0} on forward mapping properties of $I_0$ are given Section \ref{sec:proofmainXray}; after a digression on the Mellin transform and functionals generalizing the Beta function in Section \ref{sec:MellinBeta}, we then prove Theorem \ref{thm:I0starIDX} in Section \ref{sec:proofbackprojection}; Sections \ref{sec:BetaFacts} and \ref{sec:volumeFactors} then cover proofs of auxiliary lemmas (on properties of Beta functionals, and on Sasaki volume factors) stated during the proof of Theorem \ref{thm:I0starIDX}.

Finally, in Appendix \ref{sec:PFT}, we explain how the use of Melrose's pushforward and pullback theorems can be used to derive mapping properties akin to Theorems \ref{sec:proofmainXray} and \ref{thm:I0starIDX}, and discuss how such an approach overestimates the resulting index sets. 

}

\section{The double $b$-fibration} \label{sec:doubleb}

\subsection{Preliminaries: $b$-fibrations on manifolds with corners.} \label{sec:bfib}

A key theme in this paper is that the basic maps that play a role in the analysis of the X-ray transform $I_0$
(and related operators) are $b$-fibrations.  We briefly recall the definition of this class of maps and refer to 
\cite{Grieser2001,Melrose1992, Mazzeo1991} for more complete definitions and descriptions of all of the material below. 

Let $M$ and $N$ be two manifolds with corners, and enumerate the boundary hypersurfaces of these two spaces by $\{H_\alpha\}_{\alpha \in A}$
and $\{H'_\beta\}_{\beta \in B}$, respectively.   We say that $\rho_\alpha$ is a boundary defining function (bdf) for $H_\alpha$ if 
$H_\alpha = \rho_\alpha^{-1}(0)$ and $d\rho_\alpha \neq 0$ on that face. A map $F: M \to N$ is called a $b$-map if, for every $\beta \in B$,
\[
F^* \rho_\beta' =  G \cdot \rho_{\alpha_1}^{\ell_1} \ldots \rho_{\alpha_s}^{\ell_s}
\]
for some $C^\infty$ strictly positive function $G$ and some positive integers $\ell_i$.  The set of exponents which occur here, 
as $\beta$ ranges over $B$, is called the geometric data associated to $F$. It is more efficiently represented by the nonnegative
integer valued matrix $( e_F( \alpha, \beta) )_{\alpha\in A, \beta\in B}$, where $F^*\rho_\beta'$ contains the factor $\rho_\alpha^{e_F(\alpha, \beta)}$. 

On a manifold with corners $M$, the space $\mathcal V_b(M)$ of $b$-vector fields can be analogously defined as in Section \ref{sec:prelim}, as the space of all smooth vector fields which are unconstrained in the interior, and which are tangent to all boundaries of $M$. 

There is a natural vector-bundle ${}^b TM$ over $M$, called the $b$-tangent bundle, which has the property that
$\mathcal V_b(M)$ is the full space of smooth sections of ${}^b TM$. The vector fields on the right in \eqref{Vb} constitute a local basis of sections. There is a natural map ${}^b TM \to TM$ which is an isomorphism over the interior. If $p$ lies on a boundary or corner of $M$, then the nullspace of this map at $p$ is called the $b$-normal bundle ${}^b N_p M$. For example, if $M$ is a manifold with boundary, as in Sec. \ref{sec:prelim}, then ${}^bN_p M$ is the $1$-dimensional space spanned by $x\del_x$. If $M$ has corners and $p$ lies on a corner of codimension $k$, then ${}^bN_p M$ is $k$-dimensional, and naturally splits into a sum of $1$-dimensional subspaces corresponding to the $b$-normal spaces of the codimension one boundary faces intersecting at $p$.

If $F:M \to N$ is a $b$-map, then there is a naturally induced bundle map ${}^bF_*: {}^bTM \to {}^bTN$ which agrees with the ordinary differential of $F$ in the interior of $M$.   We say that $F$ is a $b$-submersion if ${}^b F_*$ is surjective at every point; similarly, $F$ is called $b$-normal if the natural restriction ${}^bF_*|_p: {}^bN_p M \to {}^bN_{f(p)} N$ is surjective.  Finally, $F$ is called a $b$-fibration if it is both a $b$-submersion and $b$-normal. 
These last two conditions are equivalent to something more familiar. Following \cite[Def. 3.9]{Grieser2001}, they are
equivalent to 
\begin{itemize}
\item[a)] If $K$ is any boundary face (i.e. hypersurface or corner), then $\text{codim}\, F(K)  \le \text{codim}\, K$; 
\item[b)] The restriction of $f$ to the interior of any boundary face $K$ is a fibration from $K^o$ to $F(K)^o$. 
\end{itemize}

\subsection{The maps $\hpi = \pi\circ\Upsilon$ and $\hF = \foot\circ\Upsilon$} \label{ssec:doubleb} 

\F{If $(M,g)$ is a non-trapping manifold with strictly convex boundary, the first-exit-time function $\tau$ defined in the introduction, when restricted to $\partial_+ SM$, can be extended smoothly to $\partial SM$ by defining 
\begin{align}
    \tilde{\tau}(x,v)=\left\{\begin{array}{rl}
	\tau(x,v), &\;\;(x,v)\in\partial_{+}SM,\\
	-\tau(x,-v), &\;\;(x,v)\in\partial_{-}SM,
    \end{array}\right.
    \label{eq:ttau}
\end{align}
see e.g. \cite[Lemma 3.2.6]{Paternain2021}. This allows to extend the map $\Upsilon$ defined in \eqref{eq:Upsilon} smoothly to}
\[
\Upsilon:\partial SM\times [0,1]\to SM, \qquad \Upsilon(x,v,u):=\varphi_{u\t(x,v)}(x,v).
\]
Note that $\Upsilon(x,v,0)=(x,v)$, $\Upsilon(x,v,1)=\alpha(x,v)$
and $\Upsilon(\alpha(x,v),u)=\Upsilon(x,v,1-u)$.  In other words, if we define $\Gamma:\partial SM\times [0,1]\to\partial SM \times [0,1]$ by
$\Gamma(x,v,u):=(\alpha(x,v),1-u)$, then $\Upsilon\circ \Gamma=\Upsilon$. 
Thus $\Upsilon$ is a 2:1 cover away from $\partial_0 SM\times[0,1]$. Each interval $\{(x,v)\} \times [0,1]$ 
with $(x,v) \in \del_0 SM$ collapses to the point $(x,v)$. 

We now claim that the pullback of the Sasaki volume form on $SM$ is given in terms of the induced volume form on $\partial  SM$ by
\begin{align}
    \Upsilon^* d\Sigma^{2d-1} = \tau\mu\ d\Sigma^{2d-2}\ du,
    \label{eq:pullback}
\end{align}
where, as we recall, $\mu(x,v) = g_x(v, \nu_x)$.  Note that the Jacobian factor vanishes at $\partial_0 SM \times [0,1]$, as expected 
since $\Upsilon$ collapses this submanifold to a lower dimensional one. To prove \eqref{eq:pullback}, suppose $f\in L^1(SM)$ and write
\begin{align*}
\int_{SM} f\ d\Sigma^{2d-1} &= \int_{\Upsilon^{-1} (SM)} (f\circ \Upsilon)\  \Upsilon^* d\Sigma^{2d-1}\\
&= \int_{\partial_+ SM} \int_0^{\tau(x,v)} f(\varphi_t(x,v))\ dt\ \mu \, d\Sigma^{2d-2} \qquad (\text{Santal\'o's formula}) \\
&\!\!\! 
= \int_{\partial_+ SM} \int_0^1 f(\Upsilon(x,v,u)) \ du\ \tau \, \mu \, d\Sigma^{2d-2} \qquad (\text{setting } t = \tau u).
\end{align*}
This establishes the claim.

We consider the map $\Upsilon$ restricted to $D = \partial_+ SM \times [0,1]$ in greater detail. 
As a first step, observe that this domain is a manifold with corners up to codimension two. Its boundary hypersurfaces are
\begin{align*}
G_0 = \partial_+ SM \times \{0\}, \qquad G_1 = \partial_+ SM \times \{1\}, \qquad G_2 = \partial_0 SM \times [0,1]. 
\end{align*}
The corresponding boundary defining functions are $u$, $1-u$ and $\mu$, respectively. The lift of $\tau$ from 
$\partial_+ SM$ to $D$ is also a bdf for $G_2$.  Finally, there are two codimension two corners: 
\begin{align*}
    F_i = G_i \cap G_2 = \partial_0 SM \times \{i\}, \qquad i=0,1.
\end{align*}
Figure \ref{fig:model} gives a schematic of the arrangement of these faces.
\begin{figure}[htpb]
    \centering
    \includegraphics[height=0.18\textheight]{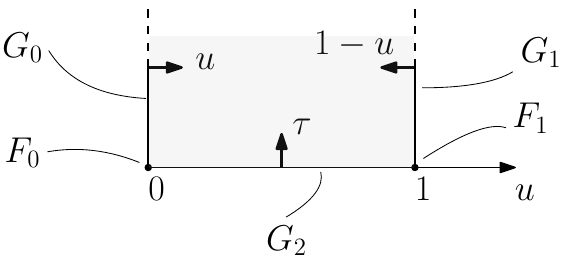}
    \caption{Local boundary model for $\Upsilon^{-1}(SM)$}
    \label{fig:model}
\end{figure}

Observe that
\begin{align*}
    \Upsilon(G_0) = \partial_+ SM, \qquad \Upsilon(G_1) = \partial_- SM,\qquad \Upsilon(G_2) = \partial_0 SM \qquad (\text{blow-down}).
\end{align*}
By \eqref{eq:pullback}, $\Upsilon|_D$ is a diffeomorphism on the interior. If $\rho$ is a bdf for $M$ equal to the geodesic 
distance to $\partial M$ near $\partial M$, then by \cite[Proposition 6.12]{Monard2020a}, 
\begin{align}
    \rho\circ \Upsilon(x,v,u) = F(x,v,u) \tau^2(x,v)u(1-u), \qquad (x,v,u)\in D,
    \label{eq:combinatorics}
\end{align}
where $F$ is a smooth strictly positive function. \resp{It is further proved in \cite[Proposition 6.12]{Monard2020a}} that $F$ extends to a map $F\colon \partial SM\times [0,1]\to \Rm$ such that

\medskip

(a) $F(\alpha(x,v),u) = F(x,v,1-u)$; 

(b) $F(x,v,0) =  \mu(x,v)/\tilde{\tau}(x,v)$ and $F(x,v,1) = \mu(\alpha(x,v))/\tilde{\tau}(\alpha(x,v))$.

(c) $F(x,v,u) = \two_x(v,v)$ when $(x,v,u)\in \partial_0 SM\times [0,1]$.

\medskip

\begin{Lemma}\label{lem:bfibrations}
    The maps $\hpi\colon D\to M$ and $\hF\colon D\to \partial_+ SM$ are $b$-fibrations with geometric data (encoding
the order of vanishing at each boundary face) given by 
    \begin{align*}
	\left.
	\begin{array}{c|c|c|c}
	    e_{\hpi} & G_0 & G_1 & G_2 \\
	    \hline 
	    \partial M & 1 & 1 & 2  
	\end{array}
	\right. \qquad 
	\left.
	\begin{array}{c|c|c|c}
	    e_{\hF} & G_0 & G_1 & G_2 \\
	    \hline 
	    \partial_0 SM & 0 & 0 & 1  
	\end{array}
	\right.
    \end{align*}
    \label{lem:bfibr}
\end{Lemma}

\begin{proof}   
({\bf First case: $\hpi$}) First note that \eqref{eq:combinatorics} already asserts that $\hpi$ is a $b$-map.
The geometric data for this map is computed directly from \eqref{eq:combinatorics}.  We now check the two 
conditions in the auxiliary characterization of $b$-fibrations \F{stated at the end of Section \ref{sec:bfib}}. Suppose that $K \in \{G_0, G_1, G_2, F_0, F_1\}$ is a boundary face. 

Condition a) is obvious since $\hpi$ maps any such $K$ onto $\del M$, which has codimension $1$. 

As for condition b), first note that $F_0$ and $F_1$ are closed (hence equal their interiors) and each is mapped 
diffeomorphically to $\partial_0 SM$ by $\Upsilon$; these are then mapped to $\partial M$ via the canonical projection 
$\partial_{0} SM = S (\partial M) \to \partial M$, which is a fibration. 

As for the boundary hypersurfaces, consider the map $\overline{\pi\circ\Upsilon} \colon G_i^o \to \partial M$, $i=0,1,2$.  
When $i = 0$, this is the composition 
\begin{align*}
\{ \mu>0\}\times \{0\}\ni (x,v,0) \stackrel{\Upsilon}{\longmapsto} (x,v) \stackrel{\pi}{\longmapsto} x;
\end{align*}
the first map is a diffeomorphism and the second is a fibration.  The analogous reasoning applies when $i = 1$.  Finally, for $i=2$, 
the map $\Upsilon$ projects onto the first factor of $\del_0 SM \times [0,1]$, and $\pi$ projects down to $\partial M$, so this
too is a fibration.

({\bf Second case: $\hF$}) The map $\hF$ is simply the projection map onto the left factor $\partial_+{SM}\times [0,1]\to \partial_+ SM$, and this
is trivially a $b$-fibration. The geometric data follows from the relation $\tau\circ\hF = \tau$. 
\end{proof}    

\subsection{The X-ray transform and the backprojection via pushforwards and pullbacks} \label{sec:compounds}

A smooth map $\phi:N_1\to N_2$ between open manifolds induces the pullback operation
\begin{align*}
    \phi^*\colon C^\infty(N_2) \to C^\infty(N_1), \qquad \phi^* f = f\circ \phi, \qquad f\in C^\infty(N_2). 
\end{align*}
If $\phi$ is proper, then we also have that $\phi^*\colon C_c^{\infty}(N_2)\to C_c^{\infty}(N_1)$.  Furthermore, if $\nu$ a smooth 
density on $N_1$, its pushforward $\phi_*\nu$ is then defined by duality
\begin{align*}
    \langle \phi_* \nu, g \rangle_{N_2} = \langle \nu, \phi^* g \rangle_{N_1}, \qquad g\in C^\infty(N_2),
\end{align*}
where $\langle \nu, f\rangle = \int f \nu$.   If $\phi$ is a submersion, then $\phi_* \nu$ can be regarded
as integration over the fibers of $\phi$, and it is well known in that case that $\phi_*$ maps $C^\infty$ to $C^\infty$. 

We now express $I_0$ and $\adj $ as weighted sequences of pushforwards and pullbacks.  At first we consider these only in the interiors of their domains of definition, but later will investigate how they affect regularity near the boundaries and corners \F{(notably in the proof of Theorem \ref{thm:I0starIDX} given in Section \ref{sec:proofbackprojection})}, which is where the $b$-fibration property will enter. \F{In addition, by virtue of \cite[Theorems 3, 5]{Melrose1992}, Lemma \ref{lem:Irewrite} factors $I_0$ and $\adj$ in a way that it becomes obvious that they map conormal spaces to conormal spaces, a necessary ingredient to our main Theorems \ref{thm:Calphamapping} and \ref{thm:I0starIDX}.}

\begin{Lemma}\label{lem:Irewrite}
Let $(M,g)$ be a convex, non-trapping Riemannian manifold with boundary, and define $\hpi = \pi\circ\Upsilon$, 
$\hF = \foot\circ\Upsilon$ as above. 
Then the maps $I_0, \adj $ defined \eqref{eq:forward}-\eqref{eq:adjoint} can be written as follows: 
    \begin{align}
	    \left(\frac{1}{\tau} I_0 f(x,v) \right)\ d\Sigma^{2d-2} &= \hF_* \left( [ \hpi^* f]d\Sigma^{2d-2} du \right), \qquad f\in C_{\resp{c}}^\infty(M^{o})	\label{eq:I0} \\[0.5ex]
	    (\adj  \resp{h})\ \dM_x  &= \hpi_* \left(\hF^* (\tau\mu \resp{h})\ d\Sigma^{2d-2}\ du\right), \qquad \resp{h}\in C^\infty( (\partial_+ SM)^{o}). \label{eq:I0star}
    \end{align}
\end{Lemma}

\begin{proof}
    When $\resp{h}(x,v,u)$ is a smooth function on $\partial SM \times [0,1]$ or $D$, we have
    \begin{align*}
	\hF_* ( \resp{h}(x,v,u)\, d\Sigma^{2d-2} du) = \left( \int_0^1 \resp{h}(x,v,u)\, du \right)\ d\Sigma^{2d-2}.
    \end{align*}
    In particular, if $\resp{h} = \hpi^* f$ for some function $f$ on $M$, then 
    \begin{align}
	\hF_* \left( [ \hpi^* f]d\Sigma^{2d-2} du \right) &= \left(\int_0^1 f(\pi(\varphi_{u\tilde \tau(x,v)}(x,v)))\ du\right)\ d\Sigma^{2d-2} \nonumber \\
	&= \left(\frac{1}{\tilde\tau(x,v)} \int_0^{\tilde \tau(x,v)} f(\pi(\varphi_t(x,v)))\ dt\right)\ d\Sigma^{2d-2}, \label{eq:bdens}
    \end{align}
    and similarly for $\Upsilon|_D$. We arrive at \eqref{eq:I0}. 

    Next, to prove \eqref{eq:I0star}, suppose $\resp{h} \in C^\infty(\partial_+ SM)$, so that its pullback $(\hF^* \resp{h})(x,v,u)$ equals $\resp{h}(x,v)$. 
    Now push forward the density $\resp{h}(x,v)\ d\Sigma^{2d-2} du$: for any \F{$f\in C^\infty(M)$}, 
    \begin{align*}
	\left\langle \hpi_* [\hF^* \resp{h}\ d\Sigma^{2d-2}\ du], f \right\rangle &= \left\langle \resp{h}(x,v)\ d\Sigma^{2d-2}\ du, \hpi^* f\right\rangle \\
	&= \int_{\partial_+ SM} \resp{h}(x,v) \int_0^1 f(\pi(\varphi_{u\tau(x,v)}(x,v)))\ du\ d\Sigma^{2d-2} \\
	&= \int_{\partial_+ SM} \frac{\resp{h}(x,v)}{\tau(x,v) \mu(x,v)} \int_{0}^{\tau(x,v)} f(\resp{\pi(}\varphi_t(x,v)\resp{)})\ dt\ \mu\ d\Sigma^{2d-2} \\
	&= \int_M \int_{S_x} \left(\frac{\resp{h}}{\tau\mu}\right)\resp{(\foot(x,v))}\ dS(v)\ f(x)\ \dM_x. 
    \end{align*}
    In particular, we arrive at the relation
    \begin{align*}
	\hpi_* [\hF^* \resp{h}\ d\Sigma^{2d-2}\ du] = \adj  \left( \frac{\resp{h}}{\tau\mu} \right)\ \dM_x,
    \end{align*}
which amounts to \eqref{eq:I0star}.
\end{proof}

\section{Special coordinates and index sets near the glancing set} 

\subsection{Preliminaries: blowups} \label{sec:blowup} 

We first recall, the notion of blowing up a $p$-submanifold in a manifold with corners
. Let $N$ be a manifold with corners, and $(x_1, \ldots, x_k, y_1, \ldots, y_{n-k})$ an `adapted' set of coordinates; this means that each $x_j$ is a bdf for a boundary hypersurface $H_j$ of $N$ and each $y_j$ is an interior variable, lying in $(-\epsilon, \epsilon)$. We say that a submanifold $Z \subset N$ is a $p$-submanifold if there exists a set of adapted coordinates such that $Z = \{(x,y): x_1 = \ldots = x_r = 0,\ y_1 = \ldots = y_s = 0\}$. 

The utility of this notion is that if $Z$ is a $p$-submanifold, then a neighborhood of the $0$-section in the inward-pointing orthant of the normal bundle $NZ$ is diffeomorphic to a neighborhood of $Z$ in $N$. 

With $Z$ as above, we define the blowup $[N;Z]$ to be the new manifold with corners obtained by replacing each point $z \in Z$ by its inward-pointing spherical normal bundle, and endowing this set with the obvious topology and the minimal $C^\infty$ structure for which the lifts of smooth functions on $N$ and spherical polar coordinates around $Z$ are all smooth.  Thus, for example, the origin $\{0\}$ is a $p$-submanifold in $\mathbb R^n$ and the blowup $[\mathbb R^n; \{0\}]$ is diffeomorphic to a half-cylinder $[0,\infty)_r \times S^{n-1}_\theta$, where $(r,\theta)$ are spherical polar coordinates on $\mathbb R^n$.

\subsection{Special coordinates on $\partial_+ SM$ near $\partial_0 SM$} \label{sec:glancingCoords}

The considerations below rely crucially on choosing a coordinate system on $\partial SM$ near $\partial_0 SM$ for which $2d-2$ 
of the coordinates are even with respect to the scattering relation, and the remaining one, which is a bdf for $\del_0 SM \subset \del_+ SM$, is odd.  
Following \cite[Theorem C.4.5]{Hoermander2007}, one can obtain such coordinates by symmetrizing and skew-symmetrizing 
any local chart near any $p\in \partial_0 SM$ with respect to $\alpha$, yielding $2(2d-1)$ functions which are of course dependent,
but then discarding a judicious half of these to obtain independent set which is the desired coordinate system.    We shall 
choose these in a slightly different and more geometric way. 

The following result does not need that $M$ be non-trapping since we only consider {\it almost glancing} geodesics, which
lie near the boundary and are never trapped. 

\begin{Lemma}[Scattering-friendly coordinates]\label{lem:scatcoords}
Let $(M,g)$ be a Riemannian manifold with strictly convex boundary, and let $\alpha$ be its scattering relation, which is well-defined and smooth in 
a neighborhood of $\partial_0 SM$. Then there exists an $\alpha$-invariant neighborhood 
$U$ of $\partial_0 SM$ in $\partial SM$ and coordinates $(t,\omega)\colon U\to \Rm\times \partial_0 SM$ satisfying 
    \begin{align*}
	(t,\omega) (\alpha(x,v)) = (-t,\omega)(x,v) \qquad \forall\, (x,v)\in U.
    \end{align*}
\end{Lemma}

When $(x,v)\in U\cap \partial_+ SM$, then $(t,\omega)(x,v)$ is interpreted as follows. Let $\varphi_t^\partial$ the geodesic 
flow on $S(\partial M) \cong \partial_0 SM$. Choose $U$ sufficiently small so that $x$ and $\pi(\alpha(x,v))$ are joined 
by a unique minimal length geodesic in $\del M$ for all $(x,v) \in U$, directed from $x$ to $\pi(\alpha(x,v))$. Then $t$ is half the length of that geodesic, and $\omega = (y,w)\in \partial_0 SM = S(\partial M)$ is the geodesic midpoint $y$ and velocity $w$ at this midpoint between $x$ and $\pi(\alpha(x,v))$ on $\del M$.

Before proving this, we first recall a result which is very close to the original definition of blowups in Section \ref{sec:blowup}.
\begin{Lemma}\label{lem:lift}
	Let $(N,g)$ \resp{be} a Riemannian manifold.
    
    (i) Let $y\in N$ and suppose that $\varepsilon>0$ is such that the exponential map $E_y \colon S_y N \times (0,\varepsilon) \to N\backslash\{y\}$ is a diffeomorphism onto its image. Then the map $E_y$ lifts to a smooth diffeomorphism 
    \begin{align*}
	\widehat{E_y}\colon S_y N\times [0,\varepsilon) \to [N,\{y\}],
    \end{align*}
which has image a neighborhood of the front face in $[N, \{y\}]$.

    (ii) Suppose that $U$ is a punctured neighborhood of the diagonal in $N\times N$ such that 
    \begin{align*}
	\Phi\colon SN\times (0,\varepsilon) \ni (x,v,t) \to (x,Exp_{x}(tv)) \in U
    \end{align*}
    is a diffeomorphism onto its image, then $\Phi$ lifts to a smooth diffeomorphism 
    \begin{align*}
	\widehat{\Phi}\colon SN\times [0,\varepsilon) \ni (x,v,t) \to (x,Exp_{x}(tv)) \in [N\times N,\Delta],
    \end{align*}
onto a neighborhood of the front face.
\end{Lemma}

\begin{proof}  For i), take exponential normal coordinates centered at $y$ and define polar coordinates $(\frac{x-y}{|x-y|}, |x-y|)$.
By definition, these are smooth on $[N,\{y\}]$ near the front face. Furthermore, $\Phi(v,t) = (\frac{v}{|v|},t)$, and this 
extends smoothly and diffeomorphically to $t=0$. This is obviously a local diffeomorphism.

The second assertion follows by simply noting that the construction in i) varies smoothly in parameters along
the diagonal.     
\end{proof}

\begin{proof}[Proof of Lemma \ref{lem:scatcoords}]
	Let $\varphi_t^\partial\colon S(\partial M) \to S(\partial M)$ \resp{be} the geodesic flow on $S(\partial M)$.  For $\varepsilon > 0$ small,
define
\begin{align*}
    U_\varepsilon^+ SM &= \{(x,v) \in \resp{\partial} SM: 0 < \mu(x,v) < \varepsilon\}, \\
    U_\varepsilon^- SM &= \{(x,v) \in \resp{\partial} SM: -\varepsilon < \mu(x,v) 
< 0\}.     
\end{align*}

Then the maps 
\begin{align*}
\Theta &\colon \partial_0 SM\times (0,\varepsilon) \ni (\omega,t) \mapsto (\pi (\varphi_{-t}^\partial (\omega)), 
\pi (\varphi_{t}^\partial (\omega))) \in \partial M\times \partial M \\
\Psi &\colon  U_\varepsilon^+SM \ni (x,v) \mapsto (\pi(x,v), \pi (\alpha(x,v))) \in \partial M\times \partial M
    \end{align*}
are diffeomorphisms onto punctured neigborhoods of the diagonal in $\partial M\times \partial M$.  Similarly, one can
define $\Theta$ and $\Psi$ on $\partial_0 SM\times (-\varepsilon,0)$ and $U_\varepsilon^- SM$, respectively.  Thus
$(\omega,t)$ is a coordinate chart on $U_\varepsilon SM := U_\varepsilon^+ SM \cup U_\varepsilon^- SM$.  Denoting by $S(x,x') = (x',x)$ the involution swapping the two factors of $\partial M\times \partial M$, then $\Theta(\omega,-t) = S(\Theta(\omega,t))$, and similarly, $\Psi(\alpha(x,v)) = S(\Psi (x,v))$ for all $(x,v)\in U_\varepsilon SM$. Note that $U_\varepsilon SM$ is a punctured neighborhood of $\partial_0 SM$ in $\partial SM$ that is stable under the scattering relation. This gives
\begin{align}
    \begin{split}
	(\omega,t) (\alpha(x,v)) &= \Theta^{-1} (\Psi (\alpha(x,v))) \\
	&= \Theta^{-1} (S(\Psi(x,v))) = (\omega,-t)(x,v), \qquad (x,v)\in U_\varepsilon SM	    
    \end{split}	
    \label{eq:proptomega}
\end{align}
and hence $(\omega,t)$ are the coordinates we are looking for, provided that they extend smoothly to $\partial_0 SM$. 

To prove this last fact, we show that $\Theta,\Psi$ extend to smooth maps
\begin{align*}
\widehat{\Theta} \colon \partial_0 SM \times [0,\varepsilon ) \to [\partial M\times \partial M,\Delta], \qquad
\widehat{\Psi} \colon \resp{U_\varepsilon^+ SM \cup \partial_0 SM} \to [\partial M\times \partial M,\Delta],
\end{align*}
where, if $(x,\left[\frac{x'-x}{|x'-x|}\right],|x'-x|)$ denote local coordinates on $[\partial M\times \partial M,\Delta]$ near the front face
($|\cdot|$ is Euclidean distance and $[v]$ denotes the equivalence class of $v$ in the spherical normal bundle of $\Delta$), we have
\begin{align*}
\widehat{\Theta}((y,w),0) = (y,[w],0) = \widehat{\Psi} ( (y,w)), \qquad (y,w)\in \partial_0 SM. 
\end{align*}

Assuming this, for the moment, then the proof follows by setting $(\omega,t) = \widehat{\Theta}^{-1}\circ \widehat{\Psi} (x,v)$ on
$U_\varepsilon^+ SM \cup \del_0 SM$. Note that $t=0$ on $\partial_0 SM$ and $\omega$ is the identity map there.  By \eqref{eq:proptomega},
this extends to $\alpha(  U_\epsilon^+ SM) \cup \del_0 SM$. 

Finally, to construct $\widehat{\Theta}$:  for each $y\in \partial M$, and in terms of coordinates 
\begin{align*}
    (x,\left[\frac{x'-x}{|x'-x|}\right],|x'-x|)    
\end{align*}
on $[\partial M\times \partial M, \Delta]$ (where $x,x'$ are written in normal exponential coordinates w.r.t. $y$), we have
\begin{align*}
\widehat{\Theta}( (y,w),t) = (y-tw, [w], 2t |w|)
\end{align*}
which clearly extends smoothly to $t=0$. \resp{The basepoint $y$ is a parameter in this construction and everything can be carried out in such a way as to depend smoothly on $y$.}

As for $\widehat{\Psi}$, Let $\widetilde{M}$ be a smooth extension of $M$, $y\in \partial M$ and $V$ a small neighborhood of $y$ in
$\widetilde{M}$. By Lemma \ref{lem:lift}.(i), the exponential map $E_y$ in $\widetilde{M}$ lifts to a smooth map $\widehat{E_y} \colon S_y \widetilde{M}\times [0,\varepsilon) \to [U,\{y\}]$. Now define the composition
\begin{align}
    \resp{U_\varepsilon^+ S_y M \cup \partial_0 S_y M} \ni v \mapsto (v, \tau(y,v)) \mapsto \widehat{\Psi}(y,v) := \widehat{E}_y (v, \tau(y,v)),
\label{eq:compo}
\end{align}
This is smooth. To show that it takes values in $[V\cap \partial M, \{y\}]$ (and not just in $[\widetilde M, \{y\}]$\resp{)}, 
first suppose that $v\in U_\varepsilon^+ S_y M$. Then $\widehat{\Psi}(y,v)$ remains in $\partial M\backslash \{y\}$. 
As $v\to w \in \partial_0 S_y M$, convexity of $M$ implies that $\tau(y,v) \to 0$ and thus for $w\in \partial_0 S_y M$, 
$\widehat{E}_y (w, \tau(y,w)) = (\frac{w}{|w|},0) \in [V\cap \partial M, \{y\}]$.  Once again, this may be carried out
smoothly in $y$. 
\end{proof}

\subsection{Special index sets and the spaces $C_{\alpha,\pm}^\infty(\partial_+ SM)$} \label{sec:indexSets}

\F{Recall the definition \eqref{eq:Cpm} of $C_{\alpha,\pm}^\infty(\partial_+ SM)$. A first observation is that the function $\tilde \tau \in C^\infty(\partial SM)$ defined in \eqref{eq:ttau} satisfies $\tilde\tau \circ \alpha = -\tilde\tau$, and hence $\tilde\tau = A_- (\tau|_{\partial_+ SM})$ and $\tau|_{\partial_+ SM}\in C_{\alpha,-}^\infty(\partial_+ SM)$. Moreover, by \cite[Lemma 6.11]{Monard2020a}, the strict convexity implies that there is a positive function $h$ on $\partial SM$ such that $\tilde \tau = h \mu$ on $\partial SM$ with $h(y,w) = \frac{2}{\two_y(w,w)}$ for $(y,w)\in \partial_0 SM$, so that $\tilde\tau$ vanishes simply at $\partial_0 SM$. With these considerations in mind, we first prove the following.}

\begin{Lemma} \label{lem:sum}  
The space of all smooth functions on $\del_+ SM$ decomposes as
    \begin{align}
	C^\infty (\partial_+ SM) = C_{\alpha,+}^\infty (\partial_+ SM) + C_{\alpha,-}^\infty (\partial_+ SM).
	\label{eq:sum}
    \end{align}
Furthermore, 
    \begin{align}
	C_{\alpha,-}^\infty (\partial_+ SM) = \tau C_{\alpha,+}^\infty (\partial_+ SM).
	\label{eq:Cminusprod}
    \end{align}
\end{Lemma}

\begin{proof} To prove the first assertion, fix any $f\in C^\infty (\partial_+ SM)$. Using a Seeley extension operator across the 
hypersurface $\partial_0 SM$, extend $f$ to $\partial_- SM$ into $\tilde f\in C^\infty(\partial SM)$. Now write 
$\tilde f = \tilde f_+ + \tilde f_-$ where $\tilde f_{\pm} = \frac{1}{2} (\tilde f \pm \alpha^* \tilde f)$. Since 
$\tilde f_\pm \in C^\infty(\partial SM)$, we have that $f_\pm := \tilde f_\pm|_{\partial_+ SM} \in C_{\alpha,\pm}^\infty$; 
since $f = f_+ + f_-$, the result follows.    

As for the second assertion, the implication $\supset$ is clear. The other inclusion follows from the fact that if 
$f\in C^\infty_{\alpha,-}$, then it vanishes to order at least $1$ at $\partial_0 SM$, so the ratio $\frac{f}{\tau}$ 
is well-defined, smooth on $\partial_+ SM$, and belongs to $C_{\alpha,+}^\infty$ since $A_+ \left(\frac{f}{\tau}\right) = \frac{A_- f}{A_-\tau}$. 
\end{proof}

The  intersection $C_{\alpha,+}^\infty\cap C_{\alpha,-}^\infty$ equals $\dot C^\infty(\partial_+ SM)$, the space of 
smooth functions vanishing to infinite order at $\partial_0 SM$. 

\paragraph{$C_\alpha^\infty$ index sets on $\partial_+ SM$.} We now define a class of polyhomogeneous functions which are a slight generalization of these classes of even and odd functions.  As above, this definition is natural given the existence of the involution $\alpha$. 

\begin{Definition}[$\alpha$-boundary defining function]
We say that a boundary defining function $\bt$ for $\del_0 SM$ in $\del SM$ is an $\alpha$-boundary defining function ($\alpha$-bdf) if $\bt\in C_{\alpha,-}^\infty(\del_+ SM)$.
\end{Definition}
Three examples of $\alpha$-bdfs are $\t$, $\mu - \mu \circ \alpha$, and the coordinate $t$ from Lemma \ref{lem:scatcoords}.

\begin{Lemma}[Characterization of $C_{\alpha,\pm}^\infty(\partial_+ SM)$]\label{lem:bdfexpansion} 
Let $\bt$ be any $\alpha$-bdf for $\partial_+ SM$.  Then for any  $f\in C^\infty_{\alpha,+} (\partial_+ SM)$, there exist smooth functions $a_{2k}$ on $\partial_0 SM$ such that 
\begin{align*}
    f \sim \sum_{k\ge 0} \bt^{2k} a_{2k}.
\end{align*}
Likewise, if $f\in C^\infty_{\alpha,-} (\partial_+ SM)$, then there exist smooth functions $a_{2k+1}$ on $\partial_0 SM$ such that 
\begin{align*}
f \sim \sum_{k\ge 0} \bt^{2k+1} a_{2k+1}.
\end{align*}    
\end{Lemma}
\begin{proof} 
If $f\in C_{\alpha,+}^\infty(\partial_+ SM)$, then $A_+ f\in C^\infty(\partial SM)$ and $A_+ f(x,v) = A_+ f(\alpha(x,v))$. Fix coordinates $(t,\omega)$ as in Lemma \ref{lem:scatcoords} in a neighborhood of $\partial_0 SM$. If $\bt$ denotes any other $\alpha$-bdf, the function $\bt/t$ belongs to $C_{\alpha,+}^\infty$ and is non-vanishing in a neighborhood of $\partial_0 SM$. In particular the jacobian of the map $(t,\omega)\mapsto (\bt,\omega)$ is non-vanishing in a neighborhood of $\partial_0 SM$, and thus we can use $(\bt,\omega)$ as local coordinates near $\partial_0 SM$. Taylor-expanding $A_+ f$ near $\{ \bt =0 \} = \partial_0 SM$, we have that $A_+ f(x,v) \sim \sum_{p\ge 0} a_p(\omega) \bt^p$ for some functions $a_p\in C^\infty(\partial_0 SM)$. Then the symmetry $A_+ f(\bt,\omega) = A_+ f(\alpha(\bt,\omega)) = A_+ f (-\bt,\omega)$ forces all odd coefficients $a_{2k+1}$ to vanish.

    The case $f\in C_{\alpha,-}^\infty$ is a direct consequence of what was just proved, equation \eqref{eq:Cminusprod}, and the fact that $\bt/\tau\in C^{\infty}_{\alpha,+}$.
\end{proof}

\begin{Definition}[$C_\alpha^\infty$-index sets on $\partial_+ SM$ \F{and $\A^*_{\phg,\alpha}(\partial_+ SM)$ spaces}]
A $C_\alpha^\infty$-index set on $\partial_+ SM$ is an index set, as defined in Section \ref{sec:prelim}, which satisfies 

\smallskip

(c') If $(z,k)\in E$, then $(z+2,k)\in E$. 

\smallskip

\noindent rather than condition (c) which distinguishes what we have been calling $C^\infty$ index sets.  

\F{With $E$ a $C_\alpha^\infty$-index set, we say that a function $f\in C^\infty( (\partial_{+}SM)^o)$ belongs to $\A_{\phg,\alpha}^E(\partial_+ SM)$ if for some $\alpha$-bdf $\bt$ for $\partial_+ SM$, there exist $a_{z,k}\in C^\infty (\partial_0 SM)$ for each $(z,k)\in E$ such that
    \begin{align*}
	f \sim \sum_{(z,k)\in E} \bt^z (\log \bt)^k a_{z,k}.
    \end{align*}}
\end{Definition}

\resp{
Note in particular that, by Lemma \ref{lem:bdfexpansion}, 
\begin{align*}
	C^\infty_{\alpha,+} (\partial_+ SM) = \A_{\phg,\alpha}^{2\Nm_0\times \{0\}} (\partial_+ SM), \qquad C^\infty_{\alpha,-} (\partial_+ SM) = \A_{\phg,\alpha}^{(2\Nm_0+1)\times \{0\}} (\partial_+ SM).
\end{align*}}

\paragraph{Coordinate invariance.}
We now show that the spaces \F{$\A_{\phg,\alpha}^E (\partial_+ SM)$ are} 
invariant under change of $\alpha$-bdf \F{and multiplication by $C_{\alpha,+}^\infty (\partial_+ SM)$ functions.}

\begin{Proposition} Fix $E$ a $C^\infty_\alpha$-index set. \\

	\F{(a)} If $f\in C^\infty( (\partial_+ SM)^o)$ has an asymptotic expansion with index set $E$ relative to some $\alpha$-bdf $\bt$, and if $\bt'$ is {\bf any} $\alpha$-bdf for $\partial_+ SM$, then 
	\begin{align*}
		f \sim \sum_{(z,k)\in E} \bt'^z (\log \bt')^k b_{z,k}.
	\end{align*}
	where the $b_{z,k} \in C^\infty(\partial_0 SM)$. \F{Hence elements of $\A_{\phg,\alpha}^E(\partial_+ SM)$ indeed have asymptotic expansions with index set at most $E$ with respect to {\em any} $\alpha$-bdf.}

	\F{(b) The set $\A_{\phg,\alpha}^E(\partial_+ SM)$ is stable under multiplication by $C_{\alpha,+}^\infty(\partial_+ SM)$ functions.}
\end{Proposition}

\begin{proof} \F{(a)} By assumption, $f$ has an expansion of the desired type, with index set $E$, with respect to one particular $\alpha$-bdf $\bt$. 

If $\bt'$ is any other $\alpha$-bdf, then $\frac{\bt}{\bt'}$ is a smooth, everywhere positive function which has only even terms in its Taylor expansion at $\del_0 SM$ with respect to either $\bt$ or $\bt'$. Using Lemma \ref{lem:bdfexpansion} and writing 
\begin{align*}
	\frac{\bt}{\bt'} \sim \sum_{p\ge 0} (\bt')^{2p} c_{2p}, \qquad c_{2p}\in C^\infty(\partial_0 SM),  \ \ c_0 > 0, 
\end{align*} 
then 
\begin{align*}
	\bt^z (\log \bt)^k  = (\bt')^z c^z_0 \left( 1 + \sum_{p\ge 1} (\bt')^{2p} \frac{c_{2p}}{c_0} \right)^z \!\!\! \left( \log \bt' + \log c_0 + \log \left( 1 + \sum_{p\ge 1} (\bt')^{2p} \frac{c_{2p}}{c_0} \right)\!\!  \right)^k.
\end{align*}
All terms are of the form $(\log \bt')^{\ell} (\bt')^{z+2p}$ for $0\le \ell \le k$ and $p\ge 0$, and these exponent pairs all lie in $E$.

\F{(b) By Lemma \ref{lem:bdfexpansion}, at the level of index sets and relative to an $\alpha$-bdf $\bt$, multiplying an element of $\A_{\phg,\alpha}^E$ by $f\in C_{\alpha,+}^\infty$ will induce operations of the form $(z,k)\to (z+2p,k)$ (where $2p\in 2\Nm_0$ denotes an exponent appearing in the Taylor expansion of $f$) for $(z,k)\in E$, and by virtue of property $(c')$, the terms generated remain within the index set $E$.}
\end{proof}

Finally, we show that \eqref{eq:combinatorics}, originally established for the geodesic distance function $\rho$
and geodesic length function $\tau$, holds for any bdf for $\del M$ and any $\alpha$-bdf for $\partial_0 SM$. 

\begin{Proposition}\label{prop:combinatorics2}
Let $(M,g)$ be convex and non-trapping, and $\rho_1$ a bdf for $\partial M$ and $\tau_1$ an $\alpha$-bdf for $\partial_0 SM$. 
Then there exists a smooth positive function $F_1$ on $\del SM\times [0,1]$ such that 
\begin{align*}
\rho_1 (\Upsilon(x,v,u)) = \tau_1^2 (x,v) u(1-u) F_1(x,v,u).
\end{align*}
for all $(x,v,u)\in \partial SM\times [0,1]$.
\end{Proposition}
\begin{proof}
This follows immediate by combining \eqref{eq:combinatorics} with the facts that $\rho/\rho_1 \in C^\infty(M)$ 
and $\tau/\tau_1\in C_{\alpha,+}^\infty(\partial_+ SM)$ are both everywhere nonvanishing.
\end{proof}

\section{Proofs of the main results} \label{sec:proofmain}

\subsection{Mapping properties of weighted X-ray transform - Proof of Theorem \ref{thm:Calphamapping} and Proposition \ref{prop:weightedI0}} \label{sec:proofmainXray}

\begin{proof}[Proof of Theorem \ref{thm:Calphamapping}] Fix the bdf $\rho$ for $M$ (which induces a bdf for $SM$), and write
    \begin{align*}
	f \sim \sum_{(z,k)\in E} \rho^z (\log \rho)^k f_{z,k}, \qquad f_{z,k}\in C^\infty(\partial SM).
    \end{align*}
    If $h \in C^\infty(\resp{\partial}SM)$,  then setting $t = \tau u$ and using that $\rho\circ \Upsilon = \tau^2 u(1-u) F$ for 
some strictly positive $F\in C^\infty(\partial_+ SM \times [0,1])$, we have
    \begin{align*}
	I^\phi (\rho^z (\log \rho)^k h) &= \tau \int_0^1 (\rho\circ\Upsilon)^z (\log \rho\circ \Upsilon)^k h(\Upsilon) \phi(\Upsilon)\ du \\
	&= \tau^{2z + 1} \int_0^1 (u(1-u))^z (F^z) (\log (\tau^2 u(1-u) F))^k h(\Upsilon) \phi(\Upsilon)\ du,
    \end{align*}
Expanding the logarithm, we arrive at 
    \begin{align}
	I^\phi (\rho^z (\log \rho)^k h) &= \tau^{1+2z} \sum_{\ell=0}^k 2^k \binom{k}{\ell} (\log \tau)^\ell a_{1+2z, \ell}(x,v),   
	\label{eq:Iw}
    \end{align}
    where, for $0\le \ell\le k$, 
    \begin{align*}
	a_{1+2z, \ell}(x,v) = \int_0^1 (u(1-u))^z F(x,v,u)^z (\log(u(1-u)F))^{k-\ell} h(\Upsilon) \phi(\Upsilon)\ du.  
    \end{align*}
Using the symmetries $\Upsilon(\alpha(x,v), u) = \Upsilon(x,v,1-u)$ and $F(\alpha(x,v),u) = F(x,v,1-u)$, we see that 
$a_{1+2z,\ell} \in C^\infty_{\alpha,+}(\partial_{+}SM)$\resp{, thus $a_{1+2z,\ell}$ contributes with an asymptotic expansion of the form $a_{1+2z,\ell} \sim \sum_{p=0}^\infty \tau^{2p} a'_{1+2z,\ell,p}(\omega)$ with $a'_{1+2z,\ell,p}\in C^\infty(\partial_0 SM)$}. Hence \eqref{eq:Iw} is an asymptotic sum involving only the terms $\tau^{1+2z+2p}(\log \tau)^\ell$, $0\le \ell\le k$ and $p\ge 0$. Since $\tau$ is an $\alpha$-bdf, we see that the corresponding index set generated takes the asserted form.

Finally, the most singular term  is when $\ell = k$,  and its coefficient at $(y,\resp{w}) \in \del_0 SM$ equals
\[
\int_0^1 \!\! 2^k ( u(1-u))^z F(y,\resp{w,u})^z h(\Upsilon) \phi(\Upsilon)\, du = 2^k B(z+1, z+1) \two_y^z(\resp{w}, \resp{w}) \phi(y,\resp{w}) h(y,\resp{w}) .
\]
This uses the identity (c) written just before Lemma~\ref{lem:bfibrations}.
\end{proof}

\begin{proof}[Proof of Proposition \ref{prop:weightedI0}] Let $\rho$ be a bdf on $M$ and $\bt$ an $\alpha$-bdf on $\partial_+ SM$. By Proposition \ref{prop:combinatorics2}, there exists a smooth positive function $F$ on $\del SM\times [0,1]$ such that 
\begin{align}
\rho (\Upsilon(x,v,u)) = \bt^2 (x,v) u(1-u) F(x,v,u).
\label{eq:combitemp}
\end{align}
Assume that $F$ is uniformly bounded by some constant $F_\infty$. Fixing some constant $\beta$, to be chosen later, we have
    \begin{align*}
	\int_{\partial_+ SM} (I_0 &(\rho^\gamma f))^2 \bt^\beta \, \mu \, d\Sigma^{2d-2} \ldots\\
	&= \int_{\partial_+ SM} \left( \int_0^\tau (f\cdot \rho^\gamma) (\varphi_t)\ dt \right)^2 \bt^{\beta} \, \mu\ d\Sigma^{2d-2} \\
	&= \int_{\partial_+ SM} \left( \int_0^1 (f\cdot \rho^\gamma)\circ \Upsilon\ du \right)^2 \bt^{\beta} \tau^2 \, \mu\ d\Sigma^{2d-2} \qquad (t=\tau u)\\
    &\le \int_{\partial_+ SM} \left( \int_0^1 (f^2 \rho^\delta)\circ \Upsilon\ du\right) \left( \int_0^1 \rho^{2\gamma-\delta}\circ \Upsilon\ du \right) \bt^{\beta} \tau^2 \mu\ d\Sigma^{2d-2}.
    \end{align*}
Now, using \eqref{eq:combitemp}, 
    \begin{align*}
	\int_0^1 \rho^{2\gamma-\delta}\circ \Upsilon\ du &= \bt^{4\gamma-2\delta} \int_0^1 (u(1-u))^{2\gamma-\delta} F(x,v,u)^{2\gamma-\delta}\ du \\
	&\le \bt^{4\gamma-2\delta} F_{\infty}^{2\gamma-\delta} \int_0^1 (u(1-u))^{2\gamma-\delta}\ du.
    \end{align*}
The $u$-integral is finite provided $2\gamma-\delta >-1$, so we arrive at 
    \begin{align*}
	\int_{\partial_+ SM} (I_0 (\rho^\gamma f))^2 \bt^\beta \mu\ d\Sigma^{2d-2} \le C \int_{\partial_+ SM} \left(\int_0^\tau (f^2 \rho^\delta) (\varphi_t)\ dt\right) \bt^{\beta+4\gamma+1-2\delta} \mu\ d\Sigma^{2d-2},
    \end{align*}
    where $C := F_\infty^{2\gamma-\delta} B(2\gamma-\delta+1, 2\gamma-\delta+1) (\tau/\bt)_\infty$. Here $B$ is the Beta function as defined in \eqref{eq:Beta}, and $\tau/\bt$ is bounded above since they are both $\alpha$-bdf's. We now set $\beta = 2\delta-4\gamma-1$ and use Santal\'o's formula to conclude that
\begin{align*}
	\int_{\partial_+ SM} (I_0 (\rho^\gamma f))^2 \bt^\beta \, \mu \, d\Sigma^{2d-2} \le C \int_{SM} f^2 \rho^{\delta}\ \resp{d}\Sigma^{2d-1} \le 
	C\, \mathrm{Vol}(\Sm^{d-1}) \int_M f^2 \rho^\delta\ \resp{\dM}.
    \end{align*}
This is the desired conclusion.
\end{proof}

\subsection{Index sets on $M$ via Mellin transform} \label{sec:MellinBeta}

We now undertake a closer examination of the integrals which give the coefficients of the terms in the expansions appearing in 
Theorem~\ref{thm:I0starIDX}, with the dual goals of showing that certain of these vanish, and of computing the leading coefficients.
A closer understanding of the Beta function and certain generalizations of it are the crucial components in this analysis.

\subsubsection{Beta-type functions}\label{sec:beta}

\paragraph{Facts about the Beta function.} The Beta function is given in terms of the Gamma function by the classical formula 
$B(z,z) = \frac{\Gamma(z)\Gamma(z)}{\Gamma(2z)}$. It is also well known that $\Gamma(z)$ has simple poles at the
negative integers, with $\lim_{z\to -n} (z+n) \Gamma(z) = \mathrm{Res}(\Gamma, -n) = \frac{(-1)^n}{n!}$, hence $B(z,z)$ also
has only simple poles at negative integers. Moreover, for $n \in \mathbb N_0$, 
\begin{align}
    \lim_{z\to -n} (z+n) B(z,z) = 2\lim_{z\to -n} \frac{(z+n) \Gamma(z) (z+n)\Gamma(z)}{2(z+n) \Gamma(2z)} = 
2 \frac{( (-1)^n/n!)^2}{(-1)^{2n}/ (2n)! } = 2 \binom{2n}{n}.
\label{eq:residueB}
\end{align}
In addition, $B(z,z)$ has a simple zero at every $z \in -\Nm_0-\frac{1}{2}$, with 
\begin{align*}
    \lim_{z\to -n-1/2} \frac{B(z,z)}{z+n+1/2} &= 2\lim_{z\to -n-1/2} \frac{\Gamma(z)^2}{(2z+2n+1)\Gamma(2z+2n+1)} \\
    &= 2\Gamma(-n-1/2)^2 \frac{1}{\lim_{\zeta\to -2n-1} \Gamma(\zeta+2n+1)(\zeta+2n+1)} \\
    &= -2 (2n+1) ! \, \Gamma(-n-1/2)^2.
\end{align*}
Since $\Gamma(-n-1/2) = \frac{(-4)^{n+1} (n+1)!}{(2n+2)!} \sqrt{\pi}$, we obtain, finally, 
\begin{align}
    \lim_{z\to -n-1/2} \frac{B(z,z)}{z+n+1/2} = -\frac{4^{2n+2}}{n+1} \, \pi \, \binom{2n+2}{n+1}^{-1}. 
    \label{eq:antiresidue}
\end{align}

\paragraph{Generalized Beta functions.}
We now consider expressions of the form 
\begin{align}
    \beta[f](z) := \int_0^1 f(u) (u(1-u))^{z-1}\ du, 
    \label{eq:betatype}
\end{align}
where $f$ is continuous on $[0,1]$ and satisfies $f(u) = f(1-u)$.  The integral defining $\beta[f]$ converges absolutely and defines a holomorphic function
on the open right half plane $\{\text{Re}(z)>0\}$. If $f$ has higher regularity, it admits a meromorphic extension to a larger half-plane: 
\begin{Proposition}\label{prop:merocont}
	If $f\in C^k([0,1])$, then $\beta[f]$ continues meromorphically to $\{\text{Re}(z)>-k\}$, with \resp{at most} simple poles at $0, -1, -2, \dots, -k+1$. 
\end{Proposition}
\begin{proof}[Proof of Proposition]
The identity $(u-1/2)^2 = 1/4 - u(1-u)$ leads to the recursion formula
\begin{align}
\beta[ (u-1/2)^2 f](z) = \frac{1}{4} \beta[f](z) - \beta[f] (z+1). 
\label{eq:idbeta1}
\end{align}

Thus if $f\in C^1([0,1])$, then 
    \begin{align}
	\beta[f](z) = \frac{2}{z} \left( (2z+1) \beta[f] (z+1) + \beta[(u-1/2)f'] (z+1) \right), \qquad \text{Re}(z)>0;
\label{eq:idbeta2}
\end{align}
this defines the continuation to $\{\text{Re}(z)>-1\}$ with a simple pole at $z=0$ and residue there 
$2\beta[f](1) + \beta[ (u-1/2)f'](1)$. This formula is proved by integration by parts: 
    \begin{align*}
	\beta[(u-1/2)f'](z+1) &= \int_0^1 (u-1/2)f'(u) (u(1-u))^z\ du \\
	&= - \int_0^1 f(u) [(u(1-u)^z + z (u-1/2)(1-2u) (u(1-u))^{z-1} ]\ du \\
	&= - \beta[f] (z+1) + 2z \beta[(u-1/2)^2 f](z) \\
	&\!\!\stackrel{\eqref{eq:idbeta1}}{=} - \beta[f] (z+1) + 2z ( \frac{1}{4} \beta[f](z) - \beta[f] (z+1)), 
    \end{align*}   
The extension to larger values of $k$ follows by repeated application of \eqref{eq:idbeta2}.    
\end{proof}

\subsubsection{Mellin functionals} \label{sec:Mellin}

\resp{Let $(M,g)$ be a smooth Riemannian manifold with boundary $\partial M$, volume form $\dM$ and boundary defining function $\rho$.} First recall that a $b$-density is a density with a `simple pole' in the bdf, i.e., a smooth multiple of $\frac{\dM}{\rho}$.
Fix a cutoff function $\chi \in C_c^\infty[0,\varepsilon)$ with $\chi\equiv 1$ near $0$. Then for any $H \in \A^t$, define the Mellin functional 
\begin{align}
	H_{\cal M}(y; z) := \int_0^\varepsilon \rho^z H(\rho,y)\chi(\rho) \frac{d\rho}{\rho}\ dA(y),
    \label{eq:Mellinh}
\end{align}
This is a (weakly) holomorphic function $\{Re(z)>-t\} \mapsto \D^\prime (\partial M)$. 

If $H$ is polyhomogeneous, a standard result relates the index set of $H$ with the poles of $H_{\cal M}$. 
\begin{Lemma}\label{lem:Mellin}
	The function $H$ is polyhomogeneous with index set $E$, where $\inf E \geq t$, if and only if $H_{\cal M}$ extends meromorphically to the entire complex plane, with a pole of order $k+1$ at $-\gamma$ 
	whenever $(\gamma,k)\in E$, and where $\eta \mapsto H_{\cal M}( \xi + i\eta, y)$ decays rapidly for each $\xi$ (for $z = \xi + i\eta$ outside the \resp{set of poles}). Specifically, given $a\in C^\infty(\partial M)$, the term $a(y) \rho^\gamma \log^k \rho$ is present in the asymptotic expansion of $H$ off of $\partial M$ if and only if the Laurent expansion of $H_{\cal M}$ about the pole $-\gamma$ takes the form $H_{\cal M}(z,y) = \frac{(-1)^k k!}{(z+\gamma)^{k+1}} (a(y) + {\cal O} (z+\gamma))$.
\end{Lemma} 

The proof is well-known, see \cite{Melrose1992}, \cite[Appendix A]{Mazzeo1991}.

\resp{
\begin{Remark}
	The rapid decay of $H_{\cal M}$ in vertical directions mirrors the conormality of $H$. In other words, if $(\rho\partial_\rho)^\ell H$ lies in some
	fixed function space for all $\ell \geq 0$, then $\zeta^\ell H_{\cal M}$ lies in the Fourier  transform of that space.   The functions encountered
	throughout this paper are usually very easy to recognize as conormal; it is their polyhomogeneous nature which is our main focus. 
	\label{decay}
\end{Remark}
}

\subsection{Proof of Theorem \ref{thm:I0starIDX}} \label{sec:proofbackprojection}

\begin{proof}[Proof of Theorem \ref{thm:I0starIDX}]
Suppose $g$ takes the form 
\begin{align*}
	g = a(\omega) \resp{\chi(\tau)} \tau^\gamma \log^k \tau, \qquad a\in C^\infty(\partial_0 SM),
\end{align*}
where $\chi\in C^\infty([0,\epsilon))$ with $1_{[0,\epsilon/3]} \le \chi \le 1_{[0,2\epsilon/3]}$ for $\epsilon$ small enough. \resp{This function is obviously conormal, then as a result of Lemma \ref{lem:Irewrite} and \cite[Theorems 3, 5]{Melrose1992}, so is $\adj g$. Then by Remark~\ref{decay} the decay of its Mellin transform in the imaginary direction is clear.} 

We now determine the index set of $\adj  g \frac{\dM}{\rho}$ and the coefficients of its expansion, at least in the range where these are local. Recall that, by Lemma \ref{lem:Irewrite},
\begin{align*}
	\adj  g \frac{\dM}{\rho} &= \hpi_\star \left( \resp{\chi(\tau)} \tau \mu g  \frac{d\Sigma^{2d-2} du}{\tau^2 u(1-u) F} \right) \\
	&= \hpi_\star \left(\resp{\chi(\tau)} \frac{\tau g}{F} \frac{\mu \, d\Sigma^{2d-2}}{\tau^2} \frac{du}{u(1-u)} \right)  \\
	&= \hpi_\star \left(\resp{\chi(\tau)}a(\omega) \frac{\tau^{\gamma+1} \log^k \tau}{F} \frac{\mu \, d\Sigma^{2d-2}}{\tau^2} \frac{du}{u(1-u)} \right)\resp{.} 
\end{align*}
By Lemma \ref{lem:Mellin}, if $h\in C^\infty(\partial M)$, we have

\begin{align}
	\langle (\adj &g)_{\cal M}, h \rangle (z) \ldots  \nonumber\\
	&= \int_{\partial M} h(y) \int_0^\varepsilon \rho^z \hpi_\star 
    \left(a(\omega) \frac{\tau^{\gamma+1} \log^k \tau}{F} \chi(\tau) \frac{\mu \, d\Sigma^{2d-2}}{\tau^2} \frac{du}{u(1-u)} \right) \nonumber \\
    &= \int_{\partial_+ SM} \int_0^1 h(y\circ \hpi) \tau^{2z+\gamma+1} \log^k \tau (u(1-u))^z F^{z-1} \chi(\tau) a(\omega) \frac{du}{u(1-u)} \frac{\mu \, d\Sigma^{2d-2}}{\tau^2} \nonumber \\
    &= \int_{\partial_+ SM} a(\omega) \tau^{2z+\gamma+1}  \log^k\tau \chi(\tau) B(\tau,\omega; z) \frac{\mu \, d\Sigma^{2d-2}}{\tau^2}, \label{eq:calculation}
\end{align}
where, for $(\tau,\omega)\in \partial_+ SM$, we have defined
\begin{align}
	B(\tau,\omega; z) := \int_0^1 h(y\ \resp{ \circ\ \pi}\circ\Upsilon(x,v,u)) F(x,v,u)^{z-1}  (u(1-u))^{z-1}\ du.
\label{eq:B}
\end{align}
Note that we used $(x,v)$ and $(\tau,\omega)$ interchangeably above. Certainly $B$ is holomorphic for $Re(z)>0$. 

The following three lemmas will be used to complete the proof of Theorem \ref{thm:I0starIDX}. 
Their proofs appear in the two subsequent sections \ref{sec:BetaFacts} and \ref{sec:volumeFactors}. 

\begin{Lemma}
\label{lem:B}
For all $z$ with $Re(z)>0$, the function $B(\cdot; z)$ belongs to $C_{\alpha,+}^\infty(\partial_+ SM)$.  Furthermore, 
$B(\cdot; z)$ continues meromorphically to $\Cm$ as a $C_{\alpha,+}^\infty(\partial_+ SM)$-valued function with \resp{at most} simple poles at $-\Nm_0$. 
\end{Lemma}

By the first part of this, 
\begin{align}
B(\tau,\omega; z) \sim \sum_{k=0}^\infty \tau^{2k} B_{k}(\omega; z), \quad \mbox{where}\quad B_k(\omega; z) = 
\frac{1}{(2k)!}\partial_\tau^{2k}|_{\tau=0} B(\tau,\omega; z).
\label{eq:Bexpansion}
\end{align}
In particular, combining \eqref{eq:B} with the fact that $F(x,v,u) = \two_x(v,v)$ and $\Upsilon(x,v,u) = (x,v)$ for all $(x,v)\in \partial_0 SM$ and $u\in [0,1]$, we directly have 
\begin{align}
B_0(\omega; z) = h(\pi(\omega)) \two(\omega)^{z-1} B(z,z).
\label{eq:B0}
\end{align}
Using the structure of the poles and zeroes of the Beta function described in Sec. \ref{sec:beta}, we can say a bit more:
\begin{Lemma}\label{lem:Bk}
	For each $k$, $B_k(\omega; z)$ has \resp{at most} simple poles at $-\Nm_0$ and \resp{at least} simple zeros at $-\Nm_0 - \frac{1}{2} - k$.
\end{Lemma}

The density $\mu \, d\Sigma^{2d-2}$ is the restriction to $\partial_+ SM$ of a smooth density on $\partial SM$ that is invariant under scattering relation. Indeed, $\alpha^* (\mu \, d\Sigma^{2d-2}) = \mu \, d\Sigma^{2d-2}$ (see \cite[Proposition 3.6.10]{Paternain2021}). 
Locally near $\partial_0 SM$, using coordinates $(\tilde\tau,\omega)$, write this density as $\ell(\tilde\tau, \omega)\ d\tilde\tau\ d\Omega$
for some smooth function $\ell$. The invariance under $\alpha^*$ and the fact that $\alpha^* d\tilde\tau = -d\tilde\tau$ and 
$\alpha^* d\Omega = d\Omega$ implies that \F{$\ell\circ \alpha = -\ell$, and hence $\ell|_{\partial_+ SM} \in C_{\alpha,-}^\infty(\partial_+ SM)$}. This implies that, upon restricting to $\partial_+ SM$ where $\tilde\tau = \tau$, 
\begin{align}
    \mu \, d\Sigma^{2d-2} \sim \sum_{p=0}^\infty s_{p} (\omega) \tau^{2p+1} d\tau d\Omega, \qquad s_{p}\in C^\infty(\partial_0 SM).
    \label{eq:musig_expansion}
\end{align}

\begin{Lemma}\label{lem:s0}
The coefficient $s_0$  is given by 
\begin{align}
s_0(\omega) = \frac{\two_y(w,w)^2}{4}, \qquad \omega = (y,w)\in \partial_0 SM,
\label{eq:s0}
\end{align}    
where $\two_y$ is the second fundamental form of $\del M$. 
\end{Lemma}

Now, combining \eqref{eq:musig_expansion} with \eqref{eq:Bexpansion} yields
\begin{align}
B(\tau,\omega; z) \frac{\mu \, d\Sigma^{2d-2}}{\tau^2} \sim \sum_{p=0}^\infty \tau^{2p} 
\left[ \sum_{q=0}^{p} B_q(\omega; z) s_{p-q}(\omega) \right] \frac{d\tau}{\tau}\ d\Omega.
\label{eq:finalExp}
\end{align}
\F{Returning to the main calculation \eqref{eq:calculation} with the asymptotic expansion \eqref{eq:finalExp} in mind,} the contribution of the $p^{\mathrm{th}}$ term of \eqref{eq:finalExp} to $\langle (\adj  g)_{\cal M}, h\rangle(z)$ equals 
\begin{align*}
    \int_{\partial_+ SM} & a(\omega) \tau^{2(z+p)+\gamma+1} \log^k\tau \left[ \sum_{q=0}^{p} B_q(\omega; z) s_{p-q}(\omega) \right] 
    \chi(\tau) \frac{d\tau}{\tau}\ d\Omega \\
    &= \F{ \underbrace{\int_{\partial_0 SM} a(\omega) \left[ \sum_{q=0}^{p} B_q(\omega; z) s_{p-q}(\omega) \right]\ d\Omega}_{A_{p,1}(z)} \cdot  \underbrace{\int_0^\varepsilon \tau^{2(z+p)+\gamma+1} \log^k \tau \, \chi(\tau) \frac{d\tau}{\tau}}_{A_{p,2}(z)}.}
\end{align*}
\F{Thus the pole structure of the $p^{\mathrm{th}}$ term is a result of the interaction of the pole and zero structures of these two factors. 
By Lemma \ref{lem:Bk}, $A_{p,1}(z)$ has at most simple poles at $-\Nm_0$ and at least simple zeros at $-\Nm_0 - p - \frac{1}{2}$.}
\F{The $A_{p,2}$ term equals}
\begin{align*}
    \F{A_{p,2}(z) = }  \ 2^{-k} \partial_z^k \left( \int_0^\varepsilon \tau^{2(z+p)+\gamma+1} \chi(\tau)\ \frac{d\tau}{\tau} \right) &= 2^{-k} \partial_z^k \left( \frac{1}{2(z+p)+\gamma+1} + E(z) \right) \\
    &= \frac{(-1)^k k!}{(2(z+p) + \gamma+1)^{k+1}} + E_{\gamma,k,p}(z),
\end{align*} 
where $E_{\gamma,k,p}(z) := 2^{-k} \partial_z^k E(z)$ is entire. \F{Hence $A_{p,2}$ has a pole of order $k+1$ at $z = -\frac{\gamma+1}{2}-p$.}

\smallskip

We now \F{split cases according to \eqref{eq:Egamk}, discussing how the zeros and poles structures of $A_{p,1}$ and $A_{p,2}$ interact.} 

\bigskip
{\bf Case $\gamma\in 2\Nm_0$.} \F{Write $\gamma = 2m$ for some $m\in \Nm_0$. For all $p\ge 0$, the product $A_{p,1}(z) A_{p,2}(z)$ inherits the simple poles of $A_{p,1}$, and the pole of $A_{p,2}$ loses one order (or disappears if the pole was originally simple) due to its interaction with one of the simple zeros of $A_{p,1}$.}

In \F{the case where $k\ge 1$}, there is a pole of order $k$ at $-\frac{(\gamma+1)}{2}-p$ \F{associated with the term $\rho^{\frac{\gamma+1}{2}+p}\log^{k-1} \rho$}, with coefficient
\begin{align*}
	\left\langle a_{(\frac{\gamma+1}{2}+p, k-1)}, h\right\rangle_{\partial M} = \frac{-k}{2^{k+1}} \int_{\partial_0 SM} a(\omega) 
	\sum_{q=0}^p  \lim_{z\to -\frac{\gamma+1}{2}-p} \left(\frac{B_q(\omega; z)}{z+ \frac{\gamma+1}{2}+p}\right) s_{p-q}(\omega)\ d\Omega.  
\end{align*}
When $p=0$, we compute
\begin{align*}
	\langle & a_{(m+1/2,  k-1)}, h\rangle_{\partial M} \ldots \\
	&= \frac{-k}{2^{k+1}} \int_{\partial_0 SM} a(\omega)  \lim_{z\to -m-1/2} \left(\frac{B_0(\omega; z)}{z+ m+ 1/2}\right) s_{0}(\omega)\ d\Omega \\
	&= \frac{-k}{2^{k+3}} \lim_{z\to -m-1/2} \left(\frac{B(z,z)}{z+ m+ 1/2}\right) \int_{\partial_0 SM} a(\omega) h(\pi(\omega)) \two(\omega)^{-m+1/2} \ d\Omega \\
	&\!\!\stackrel{\eqref{eq:antiresidue}}{=} \frac{2\pi k}{2^{k+3}} \frac{4^{2m+2}}{2m+2} \binom{2m+2}{m+1}^{-1}  \int_{\partial_0 SM} a(\omega) h(\pi(\omega)) \two(\omega)^{-m+1/2} \ d\Omega.
\end{align*}
This gives the first expression in \eqref{eq:Agk}.
\medskip

{\bf Case \F{$\gamma\in 2\Zm + 1$}.} \F{Write $\gamma = 2m-1$ for some $m\in \Zm$ and fix $p\ge 0$. 

If $p<-m$ (this is vacuous if $m>0$), then the pole of $A_{p,2}$ at $z = -m-p$ does not interact with the poles of $A_{p,1}$, and the product $A_{p,1} A_{p,2}$ has simple poles at $-\Nm_0$ and a pole of order $k+1$ at $-m-p$. The most singular coefficient in that case is unaffected by the pole interaction, thus the computation of the coefficient is similar computation to case 3 below. 

If $p\ge -m$, then $-m-p\in \Nm_0$ and the pole of $A_{p,2}$ at $z = -m-p$ interact with the poles of $A_{p,1}$ at $-\Nm_0$. The product $A_{p,1} A_{p,2}$ has simple poles at $-\Nm_0$ and a pole of order $k+2$ at $z = -m-p$. } 
The latter yields the term $\rho^{\frac{\gamma+1}{2}+p} \log^{k+1} \rho$, with coefficient
\begin{align*}
	\left\langle a_{(\frac{\gamma+1}{2}+p, k+1)}, h\right\rangle_{\partial M} \!\! = \frac{-1}{(k+1) 2^{k+1}} \int_{\partial_0 SM} \!\!\!\! a(\omega) \left[ \sum_{q=0}^p Res_{- \frac{\gamma+1}{2} - p} (B_q(\omega; \cdot)) s_{p-q}(\omega) \right] d\Omega.
\end{align*}
When \F{$m\le 0$ and} $p=0$, \F{the coefficient of the most singular term takes the form} 
\begin{align*}
	\left\langle a_{(\F{m}, k+1)}, h\right\rangle_{\partial M} &= \frac{-1}{(k+1) 2^{k+1}} \int_{\partial_0 SM} a(\omega) Res_{\F{-m}}(B_0(\omega; \cdot)) s_{0}(\omega)\ d\Omega \\
	&\stackrel{\eqref{eq:residueB}}{=} \frac{-Res_{\F{-m}}(B(z,z))}{(k+1) 2^{k+1}} \int_{\partial_0 SM} a(\omega) \two(\omega)^{\F{-m-1}}  h(\pi(\omega)) s_{0}(\omega)\ d\Omega \\
	&\stackrel{\eqref{eq:s0}}{=} \frac{-\binom{\F{2m}}{\F{m}}}{(k+1) 2^{k+2}} \int_{\partial_0 SM} a(\omega) \two(\omega)^{\F{-m+1}} h(\pi(\omega))\ d\Omega.
\end{align*}

Decomposing the Sasaki measure as $d\Omega = dA(y)\ dS_y(w)$, we obtain the middle expression in \eqref{eq:Agk}. 

\medskip
{\bf Case $\gamma\notin \F{\Nm_0\cup 2\Zm +1}$.} \F{For every $p\ge 0$, the pole of $A_{p,2}$ never interacts with the zeros or poles of $A_{p,1}$, hence the product $A_{p,1} A_{p,2}$ has simple poles at $-\Nm_0$ and a pole of order $k+1$ at
$-\frac{(\gamma+1)}{2}-p$. The latter gives the term $\rho^{\frac{\gamma+1}{2}+p} \log^{k} \rho$, with coefficient}
\begin{align*}
	\left\langle a_{(\frac{\gamma+1}{2}+p,k)}, h\right\rangle_{\partial M} = \frac{1}{2^{k+1}} \int_{\partial_0 SM} a(\omega) \left[ \sum_{q=0}^{p} B_q\left(\omega; -\frac{\gamma+1}{2}-p\right) s_{p-q}(\omega) \right]\ d\Omega.
\end{align*}
When $p=0$, 
\begin{align*}
	\left\langle a_{(\frac{\gamma+1}{2},k)}, h \right\rangle_{\partial M} &= \frac{1}{2^{k+1}} \int_{\partial_0 SM} a(w) B_0 \left(w; -\frac{\gamma+1}{2}\right) s_0 (w)\ d\Omega \\
	&= \frac{1}{2^{k+3}} \int_{\partial_0 SM} a(w) h(\pi(w)) \two(w)^{-\frac{\gamma+1}{2}-1}  (\two(w))^2\ d\Omega \\
	&= \frac{1}{2^{k+3}} B\left(-\frac{\gamma+1}{2}, -\frac{\gamma+1}{2} \right)  \\
	& \qquad \qquad \qquad \times \int_{\partial M} h(y) \int_{S_y (\partial M)} a(y,w) \two_y(w)^{\frac{-\gamma+1}{2}}\ dS_y(w)\ dA(y). 
\end{align*}
This is the bottom expression of \eqref{eq:Agk}. 

The proof of Theorem \ref{thm:I0starIDX} is complete. 
\end{proof}

\subsection{Facts about Beta functionals - Proof of Lemmas \ref{lem:B} and \ref{lem:Bk}} \label{sec:BetaFacts}

\begin{proof}[Proof of Lemma \ref{lem:B}]   Considering \eqref{eq:B}, note that the functions $\Upsilon$ and $F$ are defined 
and smooth on $\partial SM\times [0,1]$ and satisfy 
\begin{align}
F(x,v,u) = F(\alpha(x,v),1-u), \qquad \Upsilon (x,v,u) = \Upsilon (\alpha(x,v),1-u). 
\label{eq:symFU}
\end{align}
This implies that $B(\cdot; z)$ is the restriction to $\partial_+ SM$ of a function $\wtB\in C^\infty(\partial SM)$ which satisfies,
by virtue of the change of variables $u \mapsto 1-u$, 
\begin{align*}
\wtB(\alpha(x,v); z) &= \int_0^1 h(y\circ\Upsilon(\alpha(x,v),u)) F(\alpha(x,v),u)^{z-1}  \frac{du}{(u(1-u))^{1-z}} \\
&\!\!\stackrel{\eqref{eq:symFU}}{=} \int_0^1 h(y\circ\Upsilon(x,v,1-u)) F(x,v,1-u)^{z-1}  \frac{du}{ (u(1-u))^{1-z}} = \wtB(x,v; z),
\end{align*}
hence $B(\cdot; z) \in C_{\alpha,+}^\infty(\partial_+ SM)$. 

Next, the integrand in $\wtB$ is smooth in all its arguments, so we can use Proposition \ref{prop:merocont} (with an 
inconsequential parameter dependence) to extend $\wtB$ meromorphically to $\Cm$ with \resp{at most} simple poles at $-\Nm_0$. 
An inductive use of the functional relation \eqref{eq:idbeta2} shows that $\wtB(\alpha(x,v); z) = \wtB(x,v; z)$ holds on each strip 
$\{-k-1<\text{Re}(z)\le -k\}\backslash \{-k\}$. Restricting $\wtB$ to $\partial_+ SM$, we obtain the conclusion.
\end{proof}

\begin{proof}[Proof of Lemma \ref{lem:Bk}]
We find an expression for each $B_k$.  First rewrite $F(x,v,u) = \frac{\rho(\pi\circ\Upsilon(x,v,u))}{T(\Upsilon(x,v,u))}$, 
where for $(x,v)\in SM$, $T(x,v) = \tau(x,v) \tau(x,-v)$. Note that $T$ is smooth on $SM$, even on each fiber, and satisfies 
$XT = \t$, see \cite[Lemma 3.2.11]{Paternain2021}. The function 
\begin{align*}
 (x,v,t)\mapsto \rho(\pi (\varphi_t(x,v)))T(\varphi_t(x,v))
\end{align*}
is smooth in all its arguments and nonvanishing for $(x,v)$ near $\partial_0 SM$ and $t\in [0,\tau(x,v)]$. 
The function $H(x,v; z) = h(y) \left( \rho/T \right)^z$ is smooth in $(x,v)$ and entire in $z$. 
If $(x,v)\in U^+_{(0,\varepsilon)} SM$, i.e., $\mu(x,v) \in (0,\varepsilon)$, the Taylor expansion centered at $t = \frac{\tau(x,v)}{2}$ 
takes the form
\begin{align*}
H(\varphi_t(x,v); z) = \sum_{\ell=0}^{2k} \left(t-\frac{\tau(x,v)}{2} \right)^\ell H_\ell (x,v; z) + \left(t-\frac{\tau(x,v)}{2} \right)^{2k+1} 
R_{2k}(x,v,t; z),
\end{align*}
where
\begin{align*}
H_\ell(x,v; z) &:= \frac{1}{\ell !} \left. \frac{d^\ell}{dt^\ell}\right|_{t=\tau(x,v)/2} H(\varphi_{t}(x,v); z) = \frac{1}{\ell!} X^\ell H (\varphi_{\tau(x,v)/2}(x,v); z) \\
R_{2k} (x,v,t; z) &:= \frac{1}{(2k)!} \int_0^1 (1-s)^{2k} X^{2k+1} H (\varphi_{\tau/2+s(t-\tau/2)}(x,v); z)\ ds.
\end{align*}
This gives
\begin{align*}
H(\Upsilon(x,v,u); z) = \sum_{\ell=0}^{2k} \tau^\ell (u-1/2)^\ell H_\ell(x,v; z) + \tau^{2k+1} (u-1/2)^{2k+1} R_{k} (x,v,\tau u; z). 
\end{align*}
Note that $H_{2\ell}$ is entire in $z$ for every $\ell \geq 0$ and also lies in $C_{\alpha,+}^\infty$ since it is the restriction to $\partial_+ SM$ of the function $\widetilde{H}_{2\ell} (x,v; z) := 
\frac{1}{(2\ell)!}X^{2\ell}H (\varphi_{\t(x,v)/2}(x,v); z)$ which is smooth on $\partial SM$ and $\alpha$-invariant. 
    
Inserting this expansion into the expression for $B(x,v; z)$, all integrals involving odd powers of $(u-1/2)$ vanish, so we
are left with 
\begin{align*}
B(x,v; z) &= \int_0^1 H(\Upsilon(x,v,u); z) (u(1-u))^{z-1}\ du \\
&= \sum_{\ell=0}^{k} \tau^{2\ell} H_{2\ell}(x,v; z) \int_0^1 (u-1/2)^{2\ell} (u(1-u))^{z-1}\ du \\
&\qquad + \tau^{2k+1} \int_0^1 (u-1/2)^{2k+1} R_{k} (x,v,\tau u; z) (u(1-u))^z\ du.
\end{align*}
    Using the identity 
    \begin{align*}
	(u-1/2)^{2\ell} = \left( \frac{1}{4} - u(1-u) \right)^\ell = \sum_{p=0}^\ell \binom{\ell}{p} \frac{(-1)^p}{4^{\ell-p}} (u(1-u))^p,
    \end{align*}
    we obtain
    \begin{align*}
	\int_0^1 (u-1/2)^{2\ell} (u(1-u))^{z-1}\ du = \sum_{p=0}^\ell \binom{\ell}{p} \frac{(-1)^p}{4^{\ell-p}} B(z+p,z+p).
\end{align*}
   
Further expanding $H_{2\ell} (x,v; z) \sim \sum_{q=0}^\infty \tau^{2q} H_{2\ell, q} (z; \omega)$, we obtain
    \begin{align*}
	B(\tau,\omega; z) &\sim \sum_{\ell=0}^{\infty} \tau^{2\ell} \sum_{q=0}^{\infty} \tau^{2q} H_{2\ell, q}(\omega; z) \left[ \sum_{p=0}^\ell \binom{\ell}{p} \frac{(-1)^p}{4^{\ell-p}} B(z+p,z+p) \right] \\
	&\sim \sum_{k=0}^\infty \tau^{2k} \underbrace{\sum_{\ell=0}^k H_{2\ell, k-\ell}(\omega; z) \left[ \sum_{p=0}^\ell \binom{\ell}{p} \frac{(-1)^p}{4^{\ell-p}} B(z+p,z+p) \right]}_{B_k(\omega; z)}.
    \end{align*}
Interchanging the sums in $\ell$ and $p$, we arrive at
\begin{align*}
B_k(\omega; z) = \sum_{p=0}^k B(z+p, z+p) (-1)^p \left[ \sum_{\ell=p}^k \binom{\ell}{p} \frac{H_{2\ell, k-\ell}(\omega; z)}{4^{\ell-p}} \right].
\end{align*}
Using the pole and zero structure of $B(z,z)$, we see that $B(z+p,z+p)$ has simple poles at $-\Nm_0 - p$ and simple zeroes 
at $-\Nm_0 - \frac{1}{2} - p$ for each $0\le p\le k$. Furthermore, all coefficients in brackets are entire functions of $z$.
This gives the result. 
\end{proof}

\subsection{Volume factors and proof of Lemma \ref{lem:s0}} \label{sec:volumeFactors}

Let $(M^d,g)$ be a Riemannian manifold with boundary, with tangent bundle $\pi\colon TM \to M$, and recall 
the connection map $K_\nabla\colon T(TM) \to TM$: for $\xi\in T_{(x,v)} TM$, choose a curve $c(t) = (x(t),W(t)) \in TM$ 
such that $\dot c(0) = \xi$ and $W(t) \in T_{x(t)} M$, and define
\begin{align*}
K_\nabla \xi := \frac{D}{dt} W(t)|_{t=0} \qquad (D:\text{covariant differentiation}). 
\end{align*}
Writing $H(x,v) := \ker K_\nabla$ and $V(x,v) := \ker d\pi$, we have a direct sum decomposition
\begin{align*}
T_{(x,v)} TM = H(x,v)\oplus V(x,v), 
\end{align*}
where the restrictions $d\pi \colon H(x,v) \stackrel{\cong}{\longrightarrow} T_x M$ and $K_\nabla\colon V(x,v)
\stackrel{\cong}{\longrightarrow} T_x M$ are isomorphisms. Thus every $\xi\in T_{(x,v)} TM$ is uniquely 
identified with a pair $(\xi_H,\xi_V) = (d\pi(\xi), K_\nabla \xi) \in T_xM \times T_x M$, and the Sasaki metric is defined by 
\begin{align*}
    G_{(x,v)} (\xi,\eta) := g_x(\xi_H, \eta_H) + g_x(\xi_V,\eta_V), \qquad \xi,\eta\in T_{(x,v)}TM.
\end{align*}
In local coordinates, any chart $(x_1,\dots,x_d)$ on $M$ induces a chart $(x_1,\dots,x_d, y^1,\dots,y^d)$ on $TM$, where the $y^j$ are linear coordinates on each fiber, and 
\begin{align*}
H(x,v) &= \mathrm{span}\, \left\langle \delta_{x_i} := \partial_{x_i} - \Gamma_{ij}^k y^j \partial_{y^k},\ 1\le i\le d\right\rangle, \\
V(x,v) &= \mathrm{span}\, \left\langle \partial_{y^i},\ 1\le i\le d\right\rangle.
\end{align*}
Furthermore, $d\pi (\delta_{x_i}) = K_\nabla (\partial_{y^i}) = \partial_{x_i}$. 

If $(x,v)\in SM$,  the Sasaki-orthogonal decomposition is 
\begin{align*}
    T_{(x,v)} TM = T_{(x,v)} SM \stackrel{\perp}{\oplus} \Rm (0,v).
\end{align*}
Since $(0,v)\in V(x,v)$, we also have
\begin{align*}
T_{(x,v)} SM = \underbrace{\Rm X \oplus \H(x,v)}_{H(x,v)} \oplus \V(x,v), \qquad \V(x,v) := \{(0,w),\ w\in T_xM,\ g_x(w,v) = 0\},
\end{align*}
where $X = (v,0)$ and $\H(x,v) := \{(w,0),\ w\in T_xM,\ g_x(w,v) = 0\}$.  As we have written throughout the paper, $d\Sigma^{2d}$ 
and $d\Sigma^{2d-1} = \iota_{(0,v)} d\Sigma^{2d}$ denote the Sasaki volume forms on $T(TM)$ and $SM$, respectively.

A unit normal to $\partial SM$ at $(x,v) \in \partial SM$ is given by $(\nu_x,0)$, where $\nu_x$ is the unit inner normal at 
$x\in \partial M$.   We denote the Sasaki volume form on $\partial SM$ by $d\Sigma^{2d-2} = \iota_{(\nu_x,0)} d\Sigma^{2d-1}$.  

Recalling the function $\mu\colon \partial TM\to \Rm$ by $\mu(x,v) := g_x(\nu_x, v)$, we can write the tangent bundles
of the glancing manifolds $\partial_0 TM = \{\mu=0\}$ and $\partial_0 SM := \{\mu|_{SM} =0\}$ as 
\begin{align*}
    T_{(x,v)} \partial_0 TM &= \{ (\xi_H,\xi_V)\colon \xi_H \in T_x \partial M,\ g(\xi_V, \nu_x) = \two_x(v,\xi_H) \}, \\
    T_{(x,v)} \partial_0 SM &= \{ (\xi_H,\xi_V)\colon \xi_H \in T_x \partial M,\ \xi_V \in \{v\}^\perp,\ g(\xi_V, \nu_x) = \two_x(v,\xi_H) \}.
\end{align*} 
In particular, if $s$ is the shape operator of $\partial M$, then
\begin{align*}
N_0 := \frac{1}{\sqrt{g(sv,sv) + 1}} (-sv,\nu_x)
\end{align*}
is a unit normal to $\partial_0 SM$ in $\partial SM$ and to $\partial_0 TM$ in $\partial TM$.    Hence these glancing manifolds 
inherit the volume forms $\iota_{(\nu,0)} \iota_{N_0} d\Sigma^{2d}$ on $\partial_0 TM$ and 
$\iota_{(\nu,0)} \iota_{N_0} \iota_{(0,v)} d\Sigma^{2d}$ on $\partial_0 SM$.

On the other hand, identifying $\partial_0 TM$ and $\partial_0 SM$ with $T(\partial M)$ and $S(\partial M)$, respectively,
then in terms of the induced metric $i^* g$ on $\partial M$, they inherit $\partial M$-Sasaki volume forms $d\Sigma^{2d-2}_\partial$ 
and $d\Omega = \iota_{(0,v)_\partial} d\Sigma^{2d-2}_\partial$. On $T(T(\partial M))$, the Sasaki identification 
\begin{align*}
T(T (\partial M)) \approx T(\partial M) \oplus T(\partial M)
\end{align*}
is denoted $(\cdot,\cdot)_\partial$. 

We now relate the ambient and intrinsic volume forms. 
\begin{Lemma}\label{lem:Sasakirec}
On $\partial_0 TM = T(\partial M)$, 
\begin{align}
\iota_{(\nu,0)} \iota_{N_0} d\Sigma^{2d} = \sqrt{g_x(sy,sy) + 1}\ d\Sigma_\partial^{2d-2}, \qquad (x,y)\in T(\partial M). 
\label{eq:Sasakirec}
\end{align}    
\end{Lemma}
Consequently, since $(0,v) = (0,v)_\partial$, we also have 
\begin{align}
\iota_{(\nu,0)} \iota_{N_0} \iota_{(0,v)} d\Sigma^{2d} = \iota_{N_0} d\Sigma^{2d-2} = \sqrt{g_x(sv,sv)+1}\ d\Omega
\label{eq:Sasakirec2}
\end{align}
on $T(S(\partial M))$. 

\begin{proof}[Proof of Lemma \ref{lem:Sasakirec}] Given boundary normal coordinates $(x_1,\dots,x_{d-1},r)$, where $r$ is a bdf, 
we take the induced coordinates $(x_1,\dots,x_{d-1}, r, y_1, \dots, y_{d-1}, \rho)$ on $TM$. Thus any tangent vector takes the 
form $y = y^i \partial_{x_i} + \rho \partial_r$ and the unit inner normal is $\partial_r$. 

We first claim that
    \begin{align}
	(w,0)_\partial = (w, \two_x(w,y)\nu), \qquad (x,y)\in T(\partial M),\ w\in T_x(\partial M). 
	\label{eq:claimSR}
    \end{align}
Indeed, in these local coordinates, and for $1\le i\le d-1$,
    \begin{align*}
	(\partial_{x_i}, \two(\partial_{x_i},y)\nu) &= \partial_{x_i} - \sum_{j,k=1}^d \Gamma_{ij}^k y^j \partial_{y^k} + \two(\partial_{x_i},y) \partial_\rho \\
	&= \partial_{x_i} - \sum_{j,k=1}^{d-1} \Gamma_{ij}^k y^j \partial_{y^k} - \sum_{k=1}^{d} \Gamma_{id}^k \rho \partial_{y^k} - \sum_{j=1}^{d-1} \Gamma_{ij}^d y^j \partial_\rho +  \two(\partial_{x_i},y) \partial_\rho.
    \end{align*}
where $y_d = \rho$. The last two terms cancel, and the middle sum is zero since $\rho=0$ at $\del M$.  The remaining terms are simply $(\partial_{x_i},0)_\partial$, which
proves \eqref{eq:claimSR}. 

Now fix $(x,y)\in T(\partial M)$ and choose an orthonormal basis $(\be_1, \dots, \be_{d-1})$ for $T_x(\partial M)$.
Then $A_i = (\be_i, 0)_\partial$ and $B_i = (0,\be_i)_\partial$, $i \leq d-1$, is an orthonormal basis for the Sasaki metric associated to 
$i^* g$ on $T \partial M$, so 
\begin{align*}
|d\Sigma_{\partial}^{2d-2}(A_1,\dots,A_{d-1}, B_1, \dots, B_{d-1})| = 1.
\end{align*}

On the other hand, by \eqref{eq:claimSR},
\begin{align*}
A_i = (\be_i, (sy)_i \nu_x), \qquad B_i = (0,\be_i), \quad 1\le i\le d-1, 
\end{align*}
where $(sy)_i = g(\be_i, sy) = \two(\be_i,y)$). Hence,
\begin{align*}
|\iota_{(\nu,0)} \iota_{N_0} &d\Sigma^{2d} (A_1,\dots,A_{d-1}, B_1, B_{d-1})| \\
&= \frac{1}{\sqrt{g(sy,sy)+1}} |d\Sigma^{2d} (A_1,\dots,A_{d-1}, (\nu,0), B_1,\dots, B_{d-1}, (sy,-\nu))| \\
&= \frac{1}{\sqrt{g(sy,sy)+1}} \left|
\begin{array}{ccc|c}
 & & & (sy)_1 \\
& id_{d-1} & & \vdots \\
& & & (sy)_{d-1} \\
\hline
(sy)_1 & \cdots & (sy)_{d-1} & -1
\end{array}
\right| = \frac{g(sy,sy)+1}{\sqrt{g(sy,sy)+1}},
\end{align*}
by writing the determinant using the orthonormal basis 
\begin{align*}
(\be_1,0), \dots, (\be_{d-1},0), (\nu,0), (0,\be_1), \dots, (0,\be_{d-1}),(0,\nu)).
\end{align*}
This proves the result.
\end{proof}

We now use this to prove Lemma \ref{lem:s0}: 

\begin{proof}[Proof of Lemma \ref{lem:s0}] We must compute the first term in the expansion
\begin{align*}
\frac{\mu}{\t} d\Sigma^{2d-2} = s_0(\omega) d\tau d\Omega + {\cal O}(\tau^2). 
\end{align*}
On $\partial_0 SM$, $\mu/\t = \frac12 \two_y(w,w)$ and $d\mu = \frac12 \two\,  d\tau$ there. Contracting 
with $N_0$ and using \eqref{eq:Sasakirec2} gives
\begin{align*}
\frac{\two}{2} \sqrt{g_y(sw,sw)+1}\ d\Omega &= s_0 (\omega) d\tau(N_0) d\Omega = s_0(\omega) \frac{2}{\two} d\mu(N_0) d\Omega.
\end{align*}
Following the proof of \cite[Lemma 3.3.6]{Paternain2021}, for $(x,v)\in \partial SM$ and $(\xi_H,\xi_V) \in T_{(x,v)} \partial SM$,  
\begin{align*}
d\mu (\xi_H, \xi_V) = - \two_y (w,\xi_H) + g_y(\xi_V, \nu) = - g_y(sw, \xi_H) + g_y(\nu, \xi_V).
\end{align*}
In particular, 
\begin{align*}
d\mu(N_0) = \sqrt{g_y(sw,sw) + 1}, 
\end{align*}
which proves \eqref{eq:s0}. 
\end{proof}

\appendix

\section{Non-sharp mapping properties for $I_0$ and $\adj$}\label{sec:PFT}

\F{In this appendix, we explain how the pushforward and pullback theorems of Melrose \cite{Melrose1992}, combined with Lemmas \ref{lem:bfibrations} and \ref{lem:Irewrite}, help provide polyhomogeneous mapping properties akin to Theorems \ref{thm:Calphamapping} and \ref{thm:I0starIDX}, although they tend to overestimate the resulting index sets. We explain this in further detail at the end.  }

\paragraph{Polyhomogeneity on manifolds with corners.} If $N$ is a manifold with corners, we define an index family $\calE = (E_\alpha)_{\alpha \in A}$ to be a collection of index sets, one for each boundary hypersurface of $N$.   An element $u \in \A^{\calE}_{\mathrm{phg}}(N)$ is then a function with an asymptotic expansion associated to $E_\alpha$ at $H_\alpha$, and with a product type expansion (see \eqref{prodexp} for the codimension two case) at all corners. 

\paragraph{Pullback and pushforward of polyhomogeneous functions by $b$-fibrations.} \F{As seen in Lemmas \ref{lem:bfibrations} and \ref{lem:Irewrite},} the operators $I_0$ and $\adj$ can be expressed in terms of pullbacks and pushforwards by $b$-fibrations. We recall here the main facts concerning the analysis of pullbacks of polyhomogeneous functions by $b$-maps, and of pushforwards of polyhomogeneous functions by $b$-fibrations. In particular, we explain how these operations act on index sets.

For convenience, we shall only explain this in the somewhat simpler setting of a $b$-map $F\colon N_1 \to N_2$ between manifolds with corners, where $N_1$ has corners of codimension two and $N_2$ is a manifold with boundary but no corners. This reduces the amount of notation needed.  Thus let $x_1$ and $x_2$ be bdf's for the two boundary hypersurfaces of $N_1$ meeting at the corner $\{x_1 = x_2 = 0\}$ and $t$ a bdf for $\del N_2$.  Thus
\[
F^* t = G x_1^{e_1} x_2^{e_2}
\]
holds near this corner, where $G$ is smooth and everywhere positive. The geometric constants for $F$ are the elements of the pair $(e_1, e_2)$.

First, suppose that $v \in \A_{\mathrm{phg}}^E(N_2)$, and let $u = F^* v$. The pullback theorem states that $u \in \A_{\mathrm{phg}}^{f^{\#}E}(N_1)$, where $f^{\#}E = (E^{(1)}, E^{(2)})$ is a pair of $C^\infty$ index sets corresponding to the two hypersurfaces of $N_1$. These are given by 
\[
E^{(j)} = \{ (e_j z + \ell, k): (z,k) \in E\ \mbox{and}\ \ell \in \mathbb N \}.
\]
The proof is based on the identity $F^* (t^z (\log t)^k) = ( G x_1^{e_1} x_2^{e_2})^z (\log G x_1 x_2)^k$; the extra terms $(e_j z + \ell, k)$, $\ell \geq 1$, are added to ensure that $E^{(j)}$ satisfies condition (c), i.e., are $C^\infty$ index sets. 

Turning to the pushforward theorem, now suppose that $u \in \A_{\mathrm{phg}}^\calE(N_1)$, where $\calE = (E_1, E_2)$, so 
\[
u \sim \sum_{(z,k) \in E_1}  a_{z,k} \, x_1^z (\log x_1)^k\;\; \mbox{near}\ H_1, \qquad 
u \sim \sum_{(z',k') \in E_2} b_{z',k'} \, x_2^{z'} (\log x_2)^{k'}\;\; \mbox{near}\ H_2,
\]
and with a product type expansion 
\begin{equation}
u \sim \sum_{(z,k) \in E_1, (z', k') \in E_2}  c_{z,z', k, k'}\, x_1^z (\log x_1)^k x_2^{z'} (\log x_2)^{k'}
\label{prodexp}
\end{equation}
near the corner where $x_1 = x_2 = 0$.  Set
\[
E_j(e_j) = \{( z/e_j, k): (z,k) \in E_j\},
\]
and finally define the so-called 'extended union' of index sets
\begin{align*}
    f_{\#} \calE &= E_1(e_1) \overline{\cup} E_2(e_2) \\
    &=  E_1(e_1) \cup E_2(e_2) \cup \{ (z,k+k'+1): (z,k) \in E_1(e_1),\ (z, k') \in E_2(e_2) \}.    
\end{align*}
Note that if $E_1$ and $E_2$ are $C^\infty$ index sets, then so is
$E_1(e_1) \overline{\cup} E_2(e_2)$.

Pushforwards act on densities rather than functions, and we fix the $b$-densities \resp{(a $b$-density is a smooth density on the interior which extends smoothly
and nondegenerately to the closed manifold if we multiply by the bdf of each boundary hypersurface)}. Thus, for example, 
\[
d\mu = \frac{dx_1 dx_2 dy}{x_1 x_2}\ \mbox{on}\ N_1, \qquad d\nu = \frac{dt ds}{t}\ \mbox{on}\ N_2,
\]
where $y$ and $s$ are the tangential coordinates on the corner of $N_1$ and boundary of $N_2$, respectively. 
The pushforward theorem then states that if $F$ is a $b$-fibration, then
\[
    F_*( u\, d\mu ) = h \, d\nu, \qquad \text{where }\quad h \in \A^{f_{\#}\calE}_{\mathrm{phg}}(N_2).
\]
Observe that the fibres $F^{-1}(t)$ for $t > 0$ do not meet $\partial N_1$, so no positivity assumptions on the index sets are needed to ensure
that the integrals converge. 
\resp{(The pushforward theorem may of course be stated in terms of ordinary smooth densities rather than $b$-densities, but the expressions for how
the index sets combine would then involve some offsets.)} Further discussion and proofs of both of these theorems can be found in \cite{Melrose1992}.

\paragraph{Mapping properties.} Lemmas \ref{lem:bfibrations} and \ref{lem:Irewrite} express $I_0$ and $\adj $ in terms of pushforwards and pullbacks by $b$-fibrations, so we can apply the pushforward and pullback theorems explained just above, to obtain preliminary results on how these maps transform index sets. \resp{The labelling of the boundary faces of $\partial_+ SM\times [0,1]$ is the same as in Section \ref{ssec:doubleb}.} 

\begin{Lemma}\label{lem:xray}
	Let $K$ be a $C^\infty$-index set for $\partial M$. If $f\in \A_{\resp{\phg}}^K (M)$ where $\inf K >-1$, then $I_0 f \in \A_{\resp{\phg}}^{K'}(\partial_+ SM)$, where 
\begin{align*}
	K' = \{(q + 2z, p),\ q\in \Nm,\ (z,p)\in K\}.
    \end{align*}
\end{Lemma}
\begin{proof} 
	If $f\in \A_{\resp{\phg}}^K(M)$, then the pullback theorem gives
    \begin{align*}
	    \hpi^* f \in \A_{\resp{\phg}}^{\hpi^\# K}(\partial_+ SM\times [0,1]), 
    \end{align*}where
    \begin{align*}
	\hpi^\# K (G_0)  = \hpi^\# K (G_1) &= \{(q + z, p), q\in \Nm_0, (z,p)\in K\},\\
	 \hpi^\# K (G_2) &= \{(q + 2z, p), q\in \Nm_0, (z,p)\in K\}.
    \end{align*}
Next, the density $[\hpi^* f]d\Sigma^{2d-2} du$ is equivalent to the $b$-density 
    \begin{align*}
	\omega =  [\mu u(1-u) \hpi^* f]\,  \frac{d\Sigma^{2d-2}}{\mu} \frac{du}{u(1-u)}, \ \text{where} \quad \mu u(1-u) \hpi^* f \in \A_{\resp{\phg}}^L (\partial_+ SM \times [0,1]),
    \end{align*}
    \resp{and where the index set collection $L$ satisfies} 
    \begin{align*}
	L(G_0) = L(G_1) &= \{(q + z, p), q\in \Nm, (z,p)\in K\}, \\
	L(G_2) &= \{(q + 2z, p), q\in \Nm, (z,p)\in K\}.
    \end{align*}
    \resp{The integrability condition to apply the pushforward theorem becomes $\inf L(G_2) > 0$. Assuming this, and since in the equality $I_0 f \frac{d\Sigma^{2d-2}}{\tau} = \hF_* \omega$, the left side is already in $b$-density form, and the pushforward theorem immediately implies that 
    \begin{align*}
	I_0 f \in \A_{\resp{\phg}}^{\hF_\# L} (\partial_+ SM), \quad \text{where} \quad \hF_\# L = \{(q + 2z, p),\ q\in \Nm,\ (z,p)\in K\}.
    \end{align*}}
    Note that $\inf L(G_2) = 1 + 2 \inf K$, which explains the integrability condition $\inf K > -1$. 
\end{proof}

Now consider the backprojection operator
\begin{Lemma}\label{lem:I0star} If $g\in \A_{\resp{\phg}}^K (\partial_+ SM)$ for some index set $K$, then $\adj  g \in \A_{\resp{\phg}}^{K''} (M)$, where 
	\begin{align}
		K'' = (\Nm_0\times \{0\}) \, \overline{\cup} \, \left\{ \left( \frac{1+z\resp{+\ell}}{2}, p \right),\ (z,p)\in K,\ \resp{\ell\in \Nm_0} \right\}.
		\label{eq:calE}
	\end{align}   
\end{Lemma}

\begin{proof}
The index set of $\tau\mu \, g$ is 
\begin{align*}
	K' = \{ (2 + z \resp{+\ell}, p),\ (z,p)\in K,\ \resp{\ell\in \Nm_0} \}.	
\end{align*}
Thus $\hF^* (\tau\mu \, g) \in \A_{\resp{\phg}}^{\hF^\# K'}(D)$, where 
\begin{align*}
	(\hF^\# K')(G_0) = (\hF^\# K')(G_1) = \Nm_0 \times \{0\}, \qquad\text{and}\qquad (\hF^\# K')(G_2) = K'.
\end{align*}
Next consider
\begin{align*}
    \omega = \hF^* (\tau\mu g)\ d\Sigma^2\ du = \mu u(1-u) \hF^* (\tau\mu g)\ \frac{d\Sigma^2}{\mu}\ \frac{du}{u(1-u)}.
\end{align*}
As a $b$-density, it has index set 
\begin{equation}
	L(G_0) = L(G_1) = \Nm \times \{0\}, \qquad L(G_2) = \{ (3 + z\resp{+\ell},p),\ (z,p)\in K,\ \resp{\ell\in \Nm_0} \}. 
    \label{indL}
\end{equation}
The pushforward theorem now gives that $\hpi_* \omega$ is \resp{a $b$-density of the form $h (\dM/\rho)$, for some function $h\in \A_{\resp{\phg}}^{\hpi_\# L} (M)$}, where
\begin{align*}
    L' = \hpi_\# L (\partial M) &= (\rho\circ \hpi)_\# L = \widetilde{E}(F_0) \cup \widetilde {E} (F_1),
\end{align*}
and where, for $i=0,1$, 
\begin{align*}
    \widetilde{E}(F_i) = \left\{ \left( \frac{z}{e_{\rho\circ\hpi}(G_i)}, p \right), (z,p)\in L(G_i)  \right\} \overline{\cup} \left\{  \left( \frac{z}{e_{\rho\circ\hpi}(G_2)}, p \right), (z,p)\in L(G_2)  \right\}. 
\end{align*}
Using \eqref{indL} and that $e_{\rho\circ\hpi}(G_0) = e_{\rho\circ\hpi}(G_1) = 1$, $e_{\rho \circ \hpi}(G_2) = 2$, we see that 
\begin{align*}
	\widetilde{E}(F_0) = \widetilde{E}(F_1) = L'  =  (\Nm\times \{0\}) \, \overline{\cup} \, \left\{  \left( \frac{3+z \resp{+\ell}}{2}, p \right),\ (z,p)\in K,\ \resp{\ell\in\Nm_0} \right\}
\end{align*}
Notice that there is no integrability condition here because $e_{\hpi}$ is never zero. 

Now by \eqref{eq:I0star}, \resp{we have the relation $h = \rho \adj g$}. Thus $L'$ is the index set of $\rho \adj  g$, and hence \eqref{eq:calE} is the index set of $\adj  g$.
\end{proof}

\resp{

\paragraph{Discussion: Lemmas \ref{lem:xray}-\ref{lem:I0star} vs Theorems \ref{thm:Calphamapping}-\ref{thm:I0starIDX}.}

We briefly address how our main results are sharpenings of Lemmas \ref{lem:xray}-\ref{lem:I0star}. 

The mapping properties of $I_0$ are somewhat straightforward in that when using the pushforward theorem, the geometric picture does not involve boundary hypersurface interaction resulting in the creation of additional log terms. In this sense, Lemma \ref{lem:xray} almost entirely captures what needs to be captured, with the expection that the index set $K'$ (in the statement) can be made smaller: it is a $C^\infty$-index set on $\partial_+ SM$ for which many terms vanish when expressed in certain boundary defining functions. For example, for smooth integrands, Lemma \ref{lem:xray} would predict at best that $I_0 (C^\infty(M))\subset \tau C^\infty(\partial_+ SM)$, while it is known that it lands into the strict subspace $C_{\alpha,-}^\infty(\partial_+ SM)$. The latter is generally recovered by introducing $\alpha$-bdf's, $C_\alpha^\infty$-index sets on $\partial_+ SM$ and the spaces $\A_{\phg,\alpha}^E (\partial_+ SM)$.

On the mapping properties of $\adj$, an illustration of the main shortcoming of Lemma \ref{lem:I0star} is that it fails to recover \eqref{eq:PU}: if $g\in C^\infty(\partial_+ SM) = \A^{\Nm_0\times \{0\}}_{\phg}(\partial_+ SM)$, then Lemma \ref{lem:I0star} concludes that 
\begin{align*}
    \adj  (C^\infty(\partial_+ SM)) \subset C^\infty(M) + \rho\log \rho\ C^\infty(M) + \rho^{1/2} C^\infty(M).
\end{align*}
Considering parity conditions (e.g. replacing $\A^{\Nm_0\times \{0\}}_{\phg}(\partial_+ SM)$ by $\A^{2\Nm_0\times \{0\}}_{\phg,\alpha}(\partial_+ SM)$), one then arrives at 
\begin{align*}
	\adj  (C_{\alpha,+}^\infty(\partial_+ SM)) \subset C^\infty(M) + \rho^{1/2} C^\infty(M),
\end{align*}
which still does not recover \eqref{eq:PU}. This is where one need a closer examination of the proof of the pushforward theorem, and the observation, in the proof of Theorem \ref{thm:I0starIDX}, that certain expected poles of the Mellin functionals actually do not arise because of some nonobvious cancellations.
}

\paragraph{Acknowledgements.} The authors thank Gabriel P. Paternain for helpful discussions\F{, Joey Zou for pointing out a mistake in an earlier statement of Theorem \ref{thm:I0starIDX}, and the anonymous referees whose comments greatly improved the exposition of the article}. F.M was partially supported by NSF-CAREER grant DMS-1943580.  R.M. thanks both Gunther Uhlmann and Andr\'as Vasy for many years of patient guidance on inverse problems. This work was also partially supported by NSF Institutes grant DMS-1440140 when both authors were in residence at MSRI during the Fall 2019 program on Microlocal Analysis.

\bibliographystyle{siam}

\end{document}